\documentclass[reqno]{amsart}
\usepackage[T1]{fontenc}

\usepackage[left=3.1cm, right=3.1cm, bottom=4cm]{geometry}
\usepackage{setspace}
\allowdisplaybreaks[2]

\usepackage{amsmath,amssymb,amsthm}
\usepackage{hyperref}
\usepackage{doi}
\usepackage{xcolor}
\usepackage{pgf,tikz}
\usepackage{enumitem}

\usepackage{mathtools}
\usepackage{fontawesome5}

\usepackage{graphicx}
\usepackage{mathrsfs}
\usepackage{lmodern}
\usepackage{bm}

\usepackage[colorinlistoftodos,prependcaption,textsize=tiny]{todonotes}
\usepackage{xargs}
\newcommandx{\RM}[2][1=]{\todo[color=orange!30,#1]{RM: #2}}
\newcommandx{\LR}[2][1=]{\todo[color=red!30,#1]{LR: #2}}
\newcommandx{\MS}[2][1=]{\todo[color=blue!30,#1]{MS: #2}}

\DeclareMathOperator{\UP}{\mathsf{UP}_{\mathit{s}}}
\DeclareMathOperator{\NUP}{\mathsf{NUP}_{\mathit{s}}}
\newcommand\supp{\mathrm{supp}~}

\newcommand{\IM}{\mathrm{Im}}

\numberwithin{equation}{section}

\newcommand{\dmu}{\, \mathrm{d}\mu}
\newcommand{\dy}{\, \mathrm{d}y}
\newcommand{\dx}{\, \mathrm{d}x}
\newcommand{\dt}{\, \mathrm{d}t}

\newcommand{\dw}{\, \mathrm{d}w}

\newcommand{\dxi}{\, \mathrm{d}\xi}
\newcommand{\dnu}{\, \mathrm{d}\nu}

\newcommand{\sgn}{\mathrm{sgn}}
\newcommand{\st}{\mathrel{:}}

\def\R{\mathbb{R}}

\def\N{\mathbb{N}}
\def\Z{\mathbb{Z}}
\def\C{\mathbb{C}}

\def\L{\mathcal{L}}

\def\V{\mathcal{V}}

\def\H{\mathcal{H}}
\def\Sp{\mathcal{S}}

\def\D{\mathrm{D}}
\def\GL{\mathrm{GL}}

\def\Hm{\mathsf{H}}
\def\Sm{\mathsf{S}}

\renewcommand{\l}{\lambda}
\newcommand{\dist}{\operatorname{dist}}

\newtheorem{theorem}{Theorem}[section]
\newtheorem{proposition}[theorem]{Proposition}
\newtheorem{lemma}[theorem]{Lemma}
\newtheorem{corollary}[theorem]{Corollary}

\newtheorem*{questionA}{Question A}

\newtheorem*{theoremA}{Theorem A}
\newtheorem*{theoremB}{Theorem B}

\theoremstyle{definition}
\newtheorem{definition}[theorem]{Definition}
\newtheorem{remark}[theorem]{Remark}
\newtheorem{example}[theorem]{Example}

\newtheoremstyle{claimstyle}{0.75\baselineskip}{0pt}{\normalfont}{}{\normalfont}{.}{ }{}

\theoremstyle{claimstyle}
\newtheorem{claim}{Claim}
\newtheorem*{claim*}{Claim}

\newcommand{\defi}[1]{\stackrel{\mathrm{def}}{=}\, #1}
\newcommand{\hp}[1]{\left(#1\right)}
\newcommand{\hn}[1]{\left\|#1\right\|}
\newcommand{\hmm}[1]{\left|#1\right|}
\newcommand{\ha}[1]{\left\langle#1\right\rangle}

\newcommand{\hfl}[1]{\left\lfloor#1\right\rfloor}
\newcommand{\hbra}[1]{\left[#1\right]}
\newcommand{\hch}[1]{\left\{#1\right\}}

\begin{document}

\title[UP and NUP for fractional Laplacian]{Uniqueness and non-uniqueness pairs \\ for the fractional Laplacian}

\author[R.\ Motta]{Ricardo Motta}
\address[Ricardo Motta]{
  BCAM -- Basque Center for Applied Mathematics, 
  48009 Bilbao, Spain,
  \newline \phantom{\quad} \& Universidad del Pa\'{i}s Vasco / Euskal Herriko Unibertsitatea, 
  48080 Bilbao, Spain}
\email{\href{mailto:rmachado@bcamath.org}{rmachado@bcamath.org}}

\date{\today}
\thanks{The author has been supported by the Basque Government through the BERC 2026–2029 program, and by the Spanish Agencia Estatal de Investigaci\'{o}n through BCAM Severo Ochoa accreditation CEX2021-001142-S/MCIN/AEI/10.13039/501100011033, CNS2023-143893, and PID2023-146646NB-I00 funded by MICIU/AEI/10.13039/501100011033 and FEDER/EU and by ESF+.}




\begin{abstract} 
We study uniqueness and non-uniqueness pairs for the fractional Laplacian in the discrete setting. More precisely, under mild regularity assumptions on a function $f$ satisfying $f=0$ on $\Lambda$ and $(-\Delta)^s f=0$ on $M$, where $\Lambda, M \subset \mathbb{R}^d$ are discrete, we find sufficient conditions on these sets that force $f$ to vanish identically, and we provide examples in which non-uniqueness occurs. The uniqueness proofs combine decay estimates forced by discrete zero sets with analytic continuation, while the non-uniqueness results are obtained from an interpolation theorem.
\end{abstract}

\maketitle

\section{Introduction}\label{sec:1}

\subsection{Uniqueness and non-uniqueness pairs}
In this article, we are interested in understanding the following type of question:
\begin{questionA}\label{q:A}
Given $\Omega \subset \mathbb{R}^d$, spaces $\Hm$ and $\Sm$ of functions defined in $\Omega$, an operator $T:\Hm\to \Sm$, and subsets $\Lambda,M\subset\Omega$, under what conditions on $\Lambda, M$, and $\Hm$ does $f \in \Hm$ satisfying $f|_{\Lambda} = 0$ and $T(f)|_{M} = 0$ necessarily vanish identically?
\end{questionA}

If this happens, we say that $(\Lambda, M)$ is a \emph{uniqueness pair} for the operator $T$ in $\Hm$, denoted as $\mathsf{UP}_T$ for $\Hm$. Otherwise, $(\Lambda, M)$ is called a \emph{non-uniqueness pair} and is denoted as $\mathsf{NUP}_T$ for $\Hm$. This kind of question is classical, but the notion of uniqueness pairs has attracted significant attention in recent years, particularly in connection with problems arising in Fourier analysis. We define the Fourier transform for $f \in L^1(\mathbb{R}^d)$ by 
\begin{equation*} 
    \mathcal{F}(f)(\xi) = \widehat{f}(\xi) \defi \int_{\mathbb{R}^d} f(x) e^{-2\pi i \xi \cdot x}\dx. 
\end{equation*} 
When $T = \mathcal{F}$, these are known as \emph{Fourier uniqueness pairs} and are often connected to rigidity phenomena resulting from the influence of the uncertainty principle.

Perhaps the most classical example arises from the Beurling--Malliavin theorem \cite{BeurlingMalliavin1967}, which provides a description of those sets $\Lambda$ for which $(\Lambda, \R \setminus (-\delta, \delta))$ is a $\mathsf{UP}_\mathcal{F}$ or a $\mathsf{NUP}_\mathcal{F}$ for $\Hm=L^2(\R)$, given that the condition $M = \R \setminus (-\delta, \delta)$ forces the Fourier transform to be supported in $(-\delta, \delta)$. In particular, we know that $\Lambda = \mathbb{Z}$ and $M = \R \setminus(-\frac{1}{2}, \frac{1}{2})$ is a $\mathsf{UP}_\mathcal{F}$ for $\Hm$, while $\Lambda = c\mathbb{Z}$ and $M = \R \setminus(-\frac{1}{2}, \frac{1}{2})$ for any constant $c > 1$ is a $\mathsf{NUP}_\mathcal{F}$ for $\Hm$.

A remarkable result about Fourier uniqueness pairs was recently obtained by Radchenko and Viazovska \cite{RadchenkoViazovska2019}. As a consequence of an interpolation formula that they established for $\mathcal{S}_{\mathrm{even}}(\mathbb{R})$, which is the space of even functions in the Schwartz space $\mathcal{S}(\mathbb{R})$, it follows that the pair $(\Lambda, M)$ is a $\mathsf{UP}_{\mathcal{F}}$ for $\mathcal{S}_{\mathrm{even}}(\mathbb{R})$, with
\begin{equation*}
    \Lambda = M = \hch{\pm~\sqrt{n} \st n\in \Z_+}.
\end{equation*}
This result is notable for several reasons. First, both sets $\Lambda$ and $M$ are discrete, showing that rigidity phenomena may arise even under highly sparse sampling conditions. Second, their construction relies on certain arithmetic structures closely related to the solution of the sphere packing problem in dimensions $8$ and $24$ \cite{CohnKumarMillerRadchenkoViazovska2017, Viazovska2017}, and raises the question of whether this result has a fundamentally algebraic or analytic nature.

Motivated by this question, Ramos and Sousa \cite{RamosSousa2022} showed that 
\begin{equation}\label{eq:def-Z-alpha}
    \Z_\alpha \defi \hch{\pm~n^\alpha \st n\in \Z_+}
\end{equation}
forms with itself a $\mathsf{UP}_\mathcal{F}$ for $\mathcal{S}(\R)$ whenever $\alpha < 1 - \frac{\sqrt{2}}{2}$. Later, Kulikov, Nazarov, and Sodin \cite{KulikovNazarovSodin2025} extended the result of \cite{RamosSousa2022} to the full parameter range $\alpha < \frac{1}{2}$. They also introduced the more general notion of supercritical and subcritical pairs, together with corresponding $\mathsf{UP}_\mathcal{F}$ and $\mathsf{NUP}_\mathcal{F}$ results associated with these sets. All of these previous results hold in one dimension. For higher dimensions, Adve \cite{Adve2023} obtained a result for discrete subsets forming a $\mathsf{UP}_\mathcal{F}$ or $\mathsf{NUP}_\mathcal{F}$ for $\mathcal{S}(\mathbb{R}^{d})$.

Although the Fourier transform provides a classical example in which uniqueness pairs can be studied, more recent developments have shown that similar questions naturally arise in the context of partial differential equations, where the notion of uniqueness is closely related to propagation and continuation properties of solutions. A notable example in this direction is given by \textit{Heisenberg uniqueness pairs}, introduced by Hedenmalm and Montes-Rodríguez \cite{Hedenmalm2011} in the study of the Klein–Gordon equation. Analogous uniqueness results have been investigated for evolution equations. In particular, consider the Schrödinger equation
\begin{equation}\label{eq:schr-eq}
\begin{cases}
    i\partial_t u = -\Delta u+ Vu & \text{ in } \mathbb{R}\times [0,1], 
    \\ u(\cdot,0)=u_0 & \text{ on } \R.
\end{cases} 
\end{equation}
Under mild regularity assumptions on the potential $V$, Kehle and Ramos \cite{KehleRamos2022} showed that the time-one solution operator $E_1$, which maps the initial datum to the value of the corresponding solution at time $t=1$ associated with \eqref{eq:schr-eq}, admits uniqueness pairs. More precisely, for $0<\alpha<\tfrac12$ and $c_1,c_2>0$, they proved that $(\Lambda^1_\alpha(c_1),\Lambda^1_\alpha(c_2))$ is a $\mathsf{UP}_{E_1}$ for a suitable $\Hm \subset L^{1}(\mathbb{R},(1+|x|)\dx)$, where
\begin{equation}\label{eq:def-Lambda-alpha}
    \Lambda^d_{\alpha}(c) \defi \hch{c(\log|k|)^{\alpha}\frac{k}{{|k|}} \st k \in \Z^d,\, k\neq 0}.
\end{equation}
An important aspect of these results is that the uniqueness sets are discrete, and similar conclusions are obtained for certain nonlinear versions of \eqref{eq:schr-eq}.

Related phenomena have already been observed regarding unique continuation for harmonic functions in the upper half-space $\mathbb{R}^{d+1}_+$. Specifically, let $\mathcal{U}\subset C^1(\overline{\mathbb{R}^{d+1}_+})$ denote the subspace of harmonic functions, and let $D:\mathcal{U}\to L^0(\overline{\mathbb{R}^{d+1}_+})$ be defined by $D(u)=|\partial_\nu u|$, where $\partial_\nu u$ denotes the normal derivative. It has long been known that, for every nonempty open set $O\subset\mathbb{R}^d$, the pair $(O\times\{0\}, O\times\{0\})$ is a $\mathsf{UP}_D$ for $\mathcal{U}$. Such results can be traced back, for example, to Riesz \cite{Riesz1938}.

For dimension $d=1$ or for nonnegative functions in higher dimensions, the above result still holds if the assumption that the underlying set is open is replaced by the weaker requirement that it has positive Lebesgue measure (see \cite[Introduction]{AE1} for a brief discussion). Interestingly, for $d \geq 2$, there exists a set $E \subset \mathbb{R}^d$ of positive Lebesgue measure such that $(E\times\{0\},E\times\{0\})$ is a $\mathsf{NUP}_D$ for $\mathcal{U}$, as proved by the Bourgain--Wolff construction in \cite{BourgainWolff1990}.

\subsection{Main results}
We will be interested in tackling Question~\hyperref[q:A]{A} concerning the fractional Laplacian as an operator and discrete sets $\Lambda, M \subset \R^d$. For $0<s<1$, we write $(-\Delta)^s : \Hm \to \Sm$, where the function spaces $\Hm$ and $\Sm$ will be specified later, and we abbreviate $\mathsf{UP}_{(-\Delta)^s}$ and $\mathsf{NUP}_{(-\Delta)^s}$ by $\UP$ and $\NUP$, respectively.

To study uniqueness and non-uniqueness pairs on discrete sets, we first describe the geometric setup. For $0<\rho<1$, a discrete set $\Gamma\subset\R^d$ is called $\rho$-separated if there exist $A,B>0$ such that
\begin{equation}\label{eq:sep-exp}
    \dist(\gamma,\Gamma\setminus\{\gamma\}) \ge A e^{-B|\gamma|^\rho} \quad \text{for } \gamma\in\Gamma.
\end{equation} 
Moreover, given an injective function $F:\mathbb{R}^d \to \mathbb{R}^d$ and a discrete set $\Gamma \subset \mathbb{R}^d$, we measure the distribution of $\Gamma$ under the change of variables induced by $F$. Assuming that $F(\mathbb{R}^d)$ contains balls of arbitrarily large radius, consider below $M_F(\Gamma,r)$ by
\begin{equation*}
    M_F(\Gamma,r)= \hch{\frac{\#(F(\Gamma)\cap B)}{|B|}\st B\subset F(\R^d) \text{ a ball and } |B|=r} \quad \text{for } r>0,
\end{equation*}
and
\begin{equation}\label{eq:dens-with-F}
    \D_F^+(\Gamma)\defi\limsup_{r\to\infty}\sup M_F(\Gamma,r), \qquad \D_F^-(\Gamma)\defi \liminf_{r\to\infty}\inf M_F(\Gamma,r).
\end{equation}
As usual, the previous values are allowed to lie in $[0,\infty]$. Motivated by the one-dimensional case\footnote{See Remark~\ref{rem:cond-lim-der}.}, we define for $\alpha,c>0$, the changes of variables $G_{\alpha}:\R \to \R$ and $\Phi_{\alpha,c}:\R^d \to \R^d$ by
\begin{equation}\label{eq:def-G-Phi}
    G_{\alpha}(y)\defi \sgn(y)|y|^{1/\alpha}, \qquad \Phi_{\alpha,c}(x)\defi \tfrac{x}{|x|}e^{(|x|/c)^{1/\alpha}}\mathbf{1}_{\hch{|x|>0}}(x).
\end{equation}

Our first result concerns uniqueness pairs related to such functions. 
\begin{theorem}\label{thm:UP}
Let $\Hm=\Sp(\R^d)$, $0<\alpha,s<1$, and $\Lambda,M\subset \mathbb{R}^d$ be discrete. Set $\beta = \alpha^{-1}$. Then $(\Lambda, M)$ is a $\UP$ for $\Hm$ in each of the following cases:
\begin{enumerate}
    \item\label{item:1-thmUP} $d\ge 1$, and there exist $c_1,c_2>0$ such that $\D^-_{\Phi_{1,c_1}}(\Lambda)>0$ and $\D^-_{\Phi_{\beta,c_2}}(M)>0$.
    \item\label{item:2-thmUP} $d=1$, and there exist $c>0$ such that $\D^-_{\Phi_{1,c}}(\Lambda)>0$ and $\D^-_{G_{\alpha}}(M)>0$.
\end{enumerate}
\end{theorem}

As a consequence of Theorem~\ref{thm:UP}, if $f\in \Sp(\R^d)$ vanishes on $\{c_1(\log|k|){|k|^{-1}k}\}_{k\in \Z^d\setminus\hch{0}}$ and $(-\Delta)^sf$ vanishes on $\{c_2(\log|k|)^\beta{|k|^{-1}k}\}_{k\in \Z^d\setminus\hch{0}}$, where $c_1,c_2>0$ and $\beta>1$, then $f$ must be identically zero. The same conclusion holds if $f\in \Sp(\R)$ vanishes on $\{\pm c\log(n)\}_{n\in \N}$ and $(-\Delta)^sf$ vanishes on $\{\pm n^\alpha\}_{n\in \N}$, where $c>0$ and $0<\alpha<1$. Indeed, this follows directly from
\begin{equation}\label{eq:rel-dens-Z-L}
    \D_{\Phi_{\alpha,c}}^\pm(\Lambda_\alpha^d(c))=1=\D_{G_\alpha}^\pm(\Z_\alpha),
\end{equation}
where $\Lambda_\alpha^d(c)$ and $\Z_\alpha$ are the sets defined in \eqref{eq:def-Lambda-alpha} and \eqref{eq:def-Z-alpha}, respectively.

Notice that it follows by definition that if $(\Lambda,M)$ is a $\mathsf{UP}_T$ for $\Hm$, then it is also a $\mathsf{UP}_T$ for every $\Hm' \subset \Hm$. So, it is natural to consider uniqueness pairs for function spaces with the lowest possible regularity. However, there is a minor point concerning the choice of the spaces. If $T:\Hm\to\Sm$ is an operator such that $T(f)$ is prescribed on a discrete set, then $\Sm$ must be a space of functions for which pointwise evaluation is well defined, rather than a space whose elements are defined only up to almost everywhere equivalence. This means that in the discrete setting, the $\UP$ problem requires stronger regularity assumptions than those provided by $H^{s}(\mathbb{R}^d)$. Nevertheless, we remark that the first item of Theorem \ref{thm:UP} holds in the larger space $\Hm=X_{s,\sigma}(\mathbb{R}^d)$ which will be introduced in Subsection~\ref{sec:2.1}, but we state the theorem in the Schwartz space for the sake of convenience.

We next turn to non-uniqueness phenomena. In the opposite direction of the uniqueness pair situation, when $(\Lambda, M)$ is a $\mathsf{NUP}_T$ for $\Hm$, then it remains a $\mathsf{NUP}_T$ for every $\Hm'$ with $\Hm \subset \Hm'$. Hence, the smaller the space, the stronger the result about non-uniqueness. With this in mind, and motivated by \cite{KulikovNazarovSodin2025}, we introduce for $s>0$ and $0<p,q<\infty$ the Gelfand–Shilov-type spaces $\mathcal{S}_{s}^{p,q}(\mathbb{R}^d)$ as the spaces of Schwartz functions $f$ such that $(-\Delta)^sf \in \Sp(\R^d)$ and there exists $c=c_f>0$ such that
\begin{equation}\label{eq:GS-spaces}
    \sup_{x\in\R^d} |f(x)| e^{c|x|^p} + \sup_{x\in\R^d} |(-\Delta)^{s} f(x)| e^{c|x|^q} < \infty.
\end{equation}
We will proceed to identify $\NUP$ in such highly regular spaces, which are subspaces of the Schwartz class, and therefore will imply a corresponding result for $\Sp(\R^d)$.

For the sake of comparison, let us recall some corresponding $\UP$ results for open sets available in the literature. In this direction, Ghosh, Salo, and Uhlmann \cite{Ghosh2020} proved a low-regularity result showing that, for every $0<s<1$ and $r\in\mathbb{R}$, every nonempty open set $O\subset\mathbb{R}^d$ has the property that $(O, O)$ is a $\UP$ for $H^r(\mathbb{R}^d)$. Further strong unique continuation results in this setting can be found in \cite{FallFelli2014, KenigPilodPonceVega2020,Ruland2015}.

In stark contrast, there exist open sets $O \subset \mathbb{R}^d$ such that $(O, O)$ does not enjoy the uniqueness pair property for the Fourier transform. In fact, even more pathological phenomena can occur: there are nonzero functions $h\in \Sp(\R^d)$ such that both $h$ and $\widehat h$ vanish on an open neighborhood of $\Z^d$. Indeed, for $0<r<\frac{1}{2}$, take $f,g\in C_c^\infty(B(0,r))$ and $a=(\frac12,\dots,\frac12)\in \R^d$. Set $f_a(x)=f(x-a)$, $g_a(x)=g(x-a)$, and $O_r=\{x\in \R^d \st \operatorname{dist}(x,\Z^d)<\frac{\sqrt d}{2}-r\}$. Let $G(\xi)=\sum_{m\in\Z^d} g_a(m-\xi)$, which is a smooth $\Z^d$-periodic function on $\R^d$. If $\{c_k\}$ denotes the Fourier coefficients of $G$, then $\{c_k\}$ is rapidly decaying and
\begin{equation*}
    h(x)=\sum_{k\in\Z^d} c_k f_a(x-k)
\end{equation*}
belongs to $\Sp(\R^d)$. Moreover, both $h$ and $\widehat{h}$ vanish on $O_r$, given that $\widehat h(\xi)= \widehat{f_a}(\xi) G(-\xi).$

In view of this, and the results by Ghosh, Salo, and Uhlmann \cite{Ghosh2020}, the fractional Laplacian may be viewed as exhibiting a stronger form of nonlocality than the Fourier transform, at least from the perspective of their behavior on open sets. The next theorem shows that this contrast does not persist at the discrete level.
\begin{theorem}\label{thm:nup-strong}
Let $0<\rho<1$, $s>0$, and $0<p,q<1$. Let $\Lambda, M\subset\R^d$ be $\rho$-separated satisfying \eqref{eq:sep-exp}. Then $\{f\in\mathcal S_s^{p,q}(\mathbb{R}^d) \st f|_\Lambda=0 \text{ and } (-\Delta)^s f|_M=0\}$ is infinite-dimensional. In particular, $(\Lambda,M)$ is a $\NUP$ for $\mathcal S_s^{p,q}(\mathbb{R}^d)$.
\end{theorem}
For every fixed $\alpha,\beta>1$ and $c_1,c_2,s>0$, one verifies by taking $\rho=1/\min\{\alpha,\beta\}$ in Theorem~\ref{thm:nup-strong}, that there exists a nonzero Schwartz function $f$ which vanishes on $\{c_1(\log |k|)^\alpha |k|^{-1}k\}_{k\in \Z^d\setminus\hch{0}}$ and such that $(-\Delta)^s f$ vanishes on $\{c_2(\log |k|)^\beta |k|^{-1}k\}_{k\in \Z^d\setminus\hch{0}}$. Thus, in view of \eqref{eq:rel-dens-Z-L}, the logarithmic threshold on the $\Lambda$ side in Theorem~\ref{thm:UP} is sharp, in the sense that positive lower density $\D^-_{\Phi_{\beta,c}}$ for both zero sets, with $\beta>1$, does not force uniqueness. This is a major difference compared to the Fourier transform, where such a range of non-uniqueness phenomena does not hold \cite{KulikovNazarovSodin2025, RamosSousa2022}. Moreover, the mechanism behind Theorem~\ref{thm:nup-strong} is formulated as an interpolation result; see Theorem~\ref{thm:interpol-NUP} in Subsection~\ref{sec:4.3}.

When $0<s<1$, Theorems~\ref{thm:UP} and~\ref{thm:nup-strong} also admit a natural interpretation in terms of boundary unique continuation for the Caffarelli--Silvestre extension. In this formulation, prescribing the values of $f$ and $(-\Delta)^s f$ on discrete sets corresponds to prescribing, respectively, the boundary trace and the weighted normal derivative of the extension on discrete subsets of the boundary. Thus, Theorem~\ref{thm:UP} establishes a discrete boundary uniqueness principle under quantitative density assumptions, whereas Theorem~\ref{thm:nup-strong} shows that this principle fails for sufficiently sparse discrete sets, even in highly regular classes. This point of view is made explicit in Corollary~\ref{cor:UC-extension}.

We also include two further non-uniqueness constructions whose roles are different, since they show that the functions can be chosen with additional analytic structure. The first uses a Fourier-side construction on lattices, proving that the functions can be chosen bandlimited, and hence as restrictions of entire functions of exponential type.
\begin{theorem}\label{thm:BL-nup}
Let $s>0$, $M\in \GL_d(\R)$ and $\Hm = \{f \in \Sp(\R^d) \st \widehat{f} \in C_c^\infty(\R^d)\}$. Then $(M\Z^d,M\Z^d)$ is a $\NUP$ for $\Hm$. 
\end{theorem}

The second construction is specific to dimension $d=1$ and to $s=\frac12$. It uses known interpolation results in de~Branges spaces to give real-analytic functions at the power scale, as stated below.
\begin{theorem}\label{thm:debranges-nup}
Let $0<\alpha, p,q <1$ and let $\Lambda,M\subset \R$ be discrete. If both $\D^+_{G_\alpha}(\Lambda)$ and $\D^+_{G_\alpha}(M)$ are finite, then $(\Lambda, M)$ is a $\mathsf{NUP}_{\frac{1}{2}}$ for $\mathcal{S}_{1/2}^{p,q}(\mathbb{R})\cap \mathscr{H}(\R)$.
\end{theorem}

Furthermore, the proofs of Theorem~\ref{thm:UP}~\hyperref[item:2-thmUP]{$(ii)$} and Theorem~\ref{thm:debranges-nup} imply that the same pairs yield uniqueness and non-uniqueness, respectively, for the \emph{Hilbert transform}, as stated in Corollary~\ref{cor:HT} below. Here the Hilbert transform is defined, for $f \in \Sp(\R)$, as
\begin{equation*}
    {\bf H} f(x) = \mathrm{p.v.}\frac{1}{\pi}\int_{\R} \frac{f(x-t)}{t}\dt. 
\end{equation*}
\begin{corollary}\label{cor:HT}
Let $0<\alpha, p,q <1$ and $\Lambda, M \subset \R$ be discrete.
\begin{enumerate}
    \item\label{item:1-HT} If $\exists\, c>0$ such that $\D^-_{\Phi_{1,c}}(\Lambda)$ and $\D^-_{G_{\alpha}}(M)$ are positive, then $(\Lambda,M)$ is a $\mathsf{UP}_{\bf H}$ for $\Sp(\R)$.
    \item\label{item:2-HT} If both $\D^+_{G_\alpha}(\Lambda)$ and $\D^+_{G_\alpha}(M)$ are finite, then $(\Lambda, M)$ is a $\mathsf{NUP}_{\bf H}$ for $\mathcal{S}_{\bf H}^{p,q}(\mathbb{R})\cap \mathscr{H}(\R)$, where
    \begin{equation*}
        \mathcal{S}_{\bf H}^{p,q}(\mathbb{R}) \defi \hch{ f\in\mathcal{S}(\mathbb{R}) \st \exists \,c=c_f>0 \text{ such that } \sup_{x\in\mathbb{R}} |f(x)| e^{c|x|^p} + \sup_{x\in\mathbb{R}} |{\bf H} f(x)| e^{c|x|^q} < \infty}. 
    \end{equation*}
\end{enumerate}
\end{corollary}
Finally, Remarks~\ref{rem:7.1} and~\ref{rem:7.2} record how the arguments underlying Theorem~\ref{thm:UP}\hyperref[item:1-thmUP]{$(i)$} and Theorem~\ref{thm:BL-nup} extend to certain Fourier multiplier operators.

\subsection{Contents}
The paper is organized as follows. Section~\ref{sec:2} fixes the notation used throughout the paper and recalls the basic properties of the fractional Laplacian. Next, in Section~\ref{sec:3}, we prove the uniqueness pair results, namely Theorem~\ref{thm:UP} and the first item of Corollary~\ref{cor:HT}, and discuss the obstruction to frame stability. Section~\ref{sec:4} establishes some properties of Gevrey-class functions and of $\rho$-separated sets, then proves the interpolation result, Theorem~\ref{thm:interpol-NUP}, and uses it to prove Theorem~\ref{thm:nup-strong}. In Section~\ref{sec:5}, we first prove Theorem~\ref{thm:BL-nup}, then we introduce the relevant notions of density with respect to a measure, recall the necessary background on de~Branges and model spaces, and then prove Theorem~\ref{thm:debranges-nup} and item~\hyperref[item:2-HT]{$(ii)$} of Corollary~\ref{cor:HT}. Section~\ref{sec:6} provides an interpretation of these results in the context of the boundary value problem. Section~\ref{sec:7} discusses extensions to other multiplier operators. Finally, Appendix~\ref{app:A} collects an auxiliary result used throughout the paper.

\section{Preliminaries}\label{sec:2}

\subsection{Conventions and function spaces}\label{sec:2.1}
We write $f \lesssim g$ if $f \leq C g$ for some constant $C > 0$. More generally, $f \lesssim_v g$ if $C$ depends only on $v$. In general, $C$ denotes a positive constant that may change from line to line or from paragraph to paragraph in the argument. Similarly, we write $f \asymp g$ if $f\lesssim g$ and $g\lesssim f$.

We denote by $\mathbb{Z}_{+}$ the set of nonnegative integers and $\mathbb{N} = \Z_+\setminus\hch{0}$. Similarly, $\mathbb{R}_{+}$ and $\mathbb{R}_{>0}$ denote the sets of nonnegative and positive real numbers, respectively. We write $\mathbb{C}_{+}=\{z\in\mathbb{C} \st \IM (z)>0\}$ for the upper half-plane and $\mathbb{R}_{+}^{d+1}=\mathbb{R}^{d}\times\mathbb{R}_{>0}$ for the upper half-space.

We use $\#S$ for the cardinality of $S$, and $|S|$ for its Lebesgue measure, when defined. A set $S\subset \mathbb{R}^d$ is called discrete if it has no accumulation point in $\mathbb{R}^d$.

For $k \in \Z_+$, $\eta \in (0,1)$, $r\in \R$, and $\delta > 0$, we denote the space of holomorphic functions in a domain $\Omega\subset \mathbb{C}^d$ by $\mathscr{H}(\Omega)$, the space of analytic functions on $\R$ by $\mathscr{H}(\R)$, the Hölder spaces by $C^{k,\eta}(\R^d)$, and $H^r(\R^d)$ for the usual $L^2$-based Sobolev space. The Hardy space for the upper half-plane is denoted by $H^2(\C_+)$. The Gevrey class $G^\frac{1}{\eta}(U)$, for $U\subset \R^d$ an open set, is defined in Subsection~\ref{sec:4.1}. The de~Branges spaces $\H(E)$ together with the model spaces $K^2_\Theta$ are defined in Subsection~\ref{sec:5.3}. We also consider
\begin{equation*}
    L^1_\delta(\R^d) \defi \hch{ f \in L_{\operatorname{loc}}^1(\R^d) \st \int_{\R^d} \frac{|f(x)|}{1 + |x|^{d+2\delta}} \dx < \infty}.
\end{equation*}
For $0<s<1$ and\footnote{We assume that $\sigma$ is sufficiently small so that $2s + \sigma < 1$ for $0 < s < \frac{1}{2}$ and $2s - 1 + \sigma < 1$ for $\frac{1}{2} \le s < 1$. This restriction will be implicit in the remainder of the paper.} $\sigma>0$ small, let 
\begin{equation*}
C_{s,\sigma}(\R^d) \defi 
\begin{cases}
    C^{0,2s+\sigma}(\R^d), & \text{ for } 0<s<\frac{1}{2},
    \\C^{1,2s-1+\sigma}(\R^d), & \text{ for } \frac{1}{2}\leq s<1,
\end{cases}
\end{equation*}
and
\begin{equation}\label{eq:def-Xs}
    X_{s,\sigma}(\R^d) \defi L^1_{s}(\R^d) \cap C_{s,\sigma}(\R^d).
\end{equation}

\begin{remark}
It is natural to compare the space defined in \eqref{eq:def-Xs} with a Sobolev space. For $d=1$, Sobolev embedding gives $H^{\frac{1}{2}+(2s+\sigma)+\varepsilon}(\R)\subset X_{s,\sigma}(\R)$ for $\varepsilon>0$. So, for $s<3/4$ and $\sigma$ small, $X_{s,\sigma}(\R)$ is a larger space than $H^2(\R)$, since $H^2(\R)\subsetneq X_{s,\sigma}(\R)$.
\end{remark}

\subsection{The fractional Laplacian}\label{sec:2.2}

Let $s>0$. For $f\in \Sp(\R^d)$, we define the fractional Laplacian of order $s$, with the normalization used throughout this paper, by
\begin{equation}\label{eq:def-frac-Lapl}
    \mathcal F\hp{(-\Delta)^s f}(\xi) \defi |\xi|^{2s}\widehat f(\xi).
\end{equation}
This convention will be used whenever $s>0$ appears below. Moreover, \eqref{eq:def-frac-Lapl} extends for $f\in H^s(\R^d)$.

In the range $0<s<1$, the operator $(-\Delta)^s$ also admits the Caffarelli--Silvestre extension \cite{Caffarelli2007}, which interprets the fractional Laplacian as a Dirichlet--to--Neumann operator. More precisely, consider the extension $u : \mathbb{R}^{d+1}_+ \to \mathbb{C}$ solving
\begin{equation*}
\begin{cases}
    \mathrm{div}(y^{1-2s} \nabla u) = 0 &\text{ in } \mathbb{R}^{d+1}_+, 
    \\ u(\cdot,0) = f &\text{ on } \R^d,
\end{cases}
\end{equation*}
with $y>0$ denoting the extended variable. Assume in addition that $u$ is smooth and weakly vanishes as $y\to \infty$. Then the fractional Laplacian of $f$ can be recovered through the weighted normal derivative
\begin{equation}\label{eq:frac-L-limit}
    (-\Delta)^s f(x) = - C_{s} \lim_{y \to 0^+} y^{1-2s} \partial_y u(x,y),
\end{equation}
where $C_{s}$ is a constant depending only on $s$. In particular, for the Schwartz class setting, this formula \eqref{eq:frac-L-limit} coincides with the well-known singular integral representation
\begin{equation}\label{eq:def-FL-pv}
    (-\Delta)^s f(x) = c_{d,s}\mathrm{p.v.} \int_{\mathbb{R}^d} \frac{f(x) - f(w)}{|x - w|^{d+2s}}\dw,
\end{equation}
where $c_{d,s}$ is a normalization constant, which provides a direct link to the Fourier definition, as \eqref{eq:def-frac-Lapl} and \eqref{eq:def-FL-pv} are equivalent in $\mathcal{S}(\R^d)$, and we refer to \cite{Landkof1972} for a detailed proof of this equivalence. Moreover, it is well known that if $f \in \Sp(\R^d)$, then $(-\Delta)^s f$ belongs to the class
\begin{equation}\label{eq:rang-FL-S}
    \mathcal{L}_s(\R^d) = \hch{\psi \in C^\infty(\R^d) \st (1 + |x|^{d+2s}) \partial^\nu \psi(x) \in L^\infty(\R^d) \text{ for every } \nu \in \Z_+^d}.
\end{equation}

\begin{remark}\label{rem:tranla-dila}
Let $g\in\mathcal S(\R^d)$, $a\in\R^d$, $r>0$, $\sigma\in\R$, and $h(x)=r^\sigma g\hp{\frac{x-a}{r}}$. It follows immediately from the identity
$\widehat h(\xi)=r^{\sigma+d}e^{-2\pi i a\cdot \xi}\widehat g(r\xi)$ that $h\in\mathcal S(\R^d)$. In addition, if $(-\Delta)^sg$ is Schwartz, then $(-\Delta)^sh \in \Sp(\R^d)$. Moreover, for every multiindex $\nu \in \Z_+^d$
\begin{equation*}
    \partial^\nu (-\Delta)^s h(x) =r^{\sigma-2s-|\nu|}\hp{\partial^\nu(-\Delta)^s g}\hp{\frac{x-a}{r}}.
\end{equation*}
\end{remark}

Since our focus is on discrete uniqueness pairs, we require at least a notion of pointwise evaluation for the fractional Laplacian. The space $H^s(\mathbb{R}^d)$ by itself is not sufficient for this purpose. On the other hand, working in the Schwartz class is one way to guarantee it, although it is not strictly necessary. A more flexible approach is to impose suitable local regularity together with integrability conditions at infinity. Under these assumptions, as shown in \cite{Silvestre2006}, for a function $f$ in the space $X_{s,\sigma}(\mathbb{R}^d)$ (see \eqref{eq:def-Xs}), the fractional Laplacian can be defined via the pointwise formula \eqref{eq:def-FL-pv}.

The advantage of working with $X_{s,\sigma}(\R^d)$ is that we ensure not only pointwise values but also the regularity of $f$ is quantitatively transferred to $(-\Delta)^sf$. In fact, if $f \in X_{s,\sigma}(\R^d)$, then $f\in C_{s,\sigma}(\R^d)$. Hence, by this observation and \cite[Propositions~2.5 and 2.6]{Silvestre2006}, it follows that $(-\Delta)^s f$ is a Hölder continuous function with Hölder continuity exponent $\sigma$ and
\begin{equation}\label{eq:Silv-regul}
    \hbra{(-\Delta)^s f}_{C^{0,\sigma}} \lesssim_{s,\sigma} \hbra{f}_{C_{s,\sigma}}.
\end{equation}

\section{Uniqueness pairs}\label{sec:3}
\subsection{From zeros to decay}\label{sec:3.1}
The goal of this subsection is to show that the density assumption on the zero sets in Theorem~\ref{thm:UP} yields the required decay estimates. We begin by showing that a positive lower $\Phi_{\alpha,c}$-density gives an exponential-power estimate.
\begin{proposition}\label{prop:zeros-decay}
Let $\alpha,c>0$, $\Phi_{\alpha,c}:\mathbb{R}^d\to \R^d$ be as in \eqref{eq:def-G-Phi}, and let $\Gamma\subset \R^{d}$ be discrete. Assume that $\D^-_{\Phi_{\alpha,c}}(\Gamma)>0$. Then, for every $\delta \in (0,1)$ and every $x \in \mathbb{R}^d$, there exists $\gamma \in \Gamma$ satisfying
\begin{equation*}
    |x-\gamma| \lesssim e^{-\delta (|x|/c)^{1/\alpha}},
\end{equation*}
where the implicit constant is independent of $x$ and of the chosen $\gamma$.
\end{proposition}

\begin{proof}
First, $\Phi_{\alpha,c}:\R^{d}\setminus \hch{0}\to U$ is a bijection, where $U = \{u\in\R^{d} \st |u|>1\}$, with inverse $\Psi_{\alpha,c} : U \to \R^d\setminus\hch{0}$ given by
\begin{equation*}
    \Psi_{\alpha,c}(u) \defi c(\log |u|)^{\alpha}\dfrac{u}{|u|}.
\end{equation*}
By assumption, $\D^-_{\Phi_{\alpha,c}}(\Gamma)>0$. Fix $\varepsilon>0$ such that $0<\varepsilon<\D^-_{\Phi_{\alpha,c}}(\Gamma)$. Using the definition of lower density, there exists $r_1>0$ such that for every $r\ge r_1$
\begin{equation*}
    \inf \hch{\frac{\#(\Phi_{\alpha,c}(\Gamma)\cap B)}{|B|}\st B\subset \Phi_{\alpha,c}(\R^d) \text{ a ball and } |B|=r}\geq \varepsilon.
\end{equation*}
Thus every ball $B\subset \Phi_{\alpha,c}(\R^{d})$ with $|B|\geq r_1$ satisfies $\#(\Phi_{\alpha,c}(\Gamma)\cap B)\ge \varepsilon |B|.$ Choose $R>0$ so large that $\omega_{d}R^{d}\ge r_1$ and $\varepsilon\omega_{d}R^{d} \ge1$, where $\omega_{d}=|B(0,1)|$. Let $B\subset \Phi_{\alpha,c}(\R^{d})$ be any ball of radius $R$. Then
$|B|=\omega_{d}R^{d}\geq r_1$, and hence $\#(\Phi_{\alpha,c}(\Gamma)\cap B)\ge \varepsilon\omega_{d}R^{d}\geq 1$. Thus every ball of radius $R$ contained in $\Phi_{\alpha,c}(\R^{d})$ intersects $\Phi_{\alpha,c}(\Gamma)$. Now let $u\in \Phi_{\alpha,c}(\R^{d})$ satisfy $|u|>R+1$. Then $B(u,R)\subset \Phi_{\alpha,c}(\R^{d})$ and
\begin{equation}\label{eq:ball-inter-G}
    B(u,R)\cap \Phi_{\alpha,c}(\Gamma)\neq \emptyset \quad \text{for all }  |u|>R+1.
\end{equation}

Fix $\delta\in(0,1)$. Let $x\in\R^d$ with $|\Phi_{\alpha,c}(x)|>\max\hch{2R+2,R+e^\alpha}$, and set $u=\Phi_{\alpha,c}(x)$. Then $B(u,R)\subset U$, and \eqref{eq:ball-inter-G} yields a point $v\in B(u,R)\cap \Phi_{\alpha,c}(\Gamma)$. Thus $v=\Phi_{\alpha,c}(\gamma)$ for some $\gamma\in\Gamma$, and $|u-\Phi_{\alpha,c}(\gamma)|\le R$. Therefore, using the inverse function,
\begin{equation}\label{eq:rang-Psi}
    |x-\gamma|=|\Psi_{\alpha,c}(u)-\Psi_{\alpha,c}(v)|.
\end{equation}
The goal now is to apply the mean value inequality. For $|w|>2$, write $\Psi_{\alpha,c}(w)=\phi(|w|)\frac{w}{|w|}$, where $\phi(r)=c(\log r)^{\alpha}$. For a smooth radial map of the form $G(w)=\phi(|w|)w/|w|$, one has the standard decomposition of the differential into radial and tangential parts
\begin{equation*}
    DG(w)h = \phi'(|w|)\ha{ h,\frac{w}{|w|}} \frac{w}{|w|} + \frac{\phi(|w|)}{|w|} \hp{h-\ha{h,\frac{w}{|w|}} \frac{w}{|w|}}.
\end{equation*}
Applying this to $\Psi_{\alpha,c}$, we obtain $\phi'(r)=c\alpha (\log r)^{\alpha-1}/r$, and hence for $|w|>2$
\begin{equation}\label{eq:bd-Der}
    \hn{D\Psi_{\alpha,c}(w)} \leq |\phi'(|w|)|+\frac{\phi(|w|)}{|w|} = \frac{c\alpha(\log |w|)^{\alpha-1}}{|w|} + \frac{c(\log |w|)^\alpha}{|w|} \lesssim_{\alpha, c} \frac{(\log |w|)^\alpha}{|w|}.
\end{equation}
Moreover, every point $w$ on the convex segment $[u,v]$ satisfies $|w|\ge |u|-R\geq \max\hch{R+2,e^\alpha}$, so \eqref{eq:bd-Der} is valid along the whole segment. By the mean value inequality applied to $\Psi_{\alpha,c}$ on $[u,v]$,
\begin{equation*}
    |\Psi_{\alpha,c}(u)-\Psi_{\alpha,c}(v)| \leq |u-v| \sup_{w\in[u,v]}\hn{D\Psi_{\alpha,c}(w)} \lesssim_{\alpha, c} R\sup_{w\in[u,v]} \frac{(\log |w|)^\alpha}{|w|},
\end{equation*}
where the last inequality follows from $|u-v|\le R$ and \eqref{eq:bd-Der}. Since $|w| \ge |u| - R$ for $w \in [u, v]$ and the function $h(y) = (\log y)^\alpha/ y$ is decreasing in $[e^\alpha,\infty)$, it follows by \eqref{eq:rang-Psi}, after possibly enlarging the implicit constants, that
\begin{equation}\label{eq:log-bound}
    |x-\gamma| =|\Psi_{\alpha,c}(u)-\Psi_{\alpha,c}(v)|\lesssim_{R,\alpha, c}\frac{(\log |u|)^\alpha}{|u|}.
\end{equation}
Since $|u|=e^{(|x|/c)^{1/\alpha}}$, if we write $t=(|x|/c)^{1/\alpha}$, then $\log |u|=t$, and \eqref{eq:log-bound} becomes $|x-\gamma|\lesssim_{\alpha, c, R} t^\alpha e^{-t}$. For $0<\delta<1$, the function $t^\alpha e^{-(1-\delta)t}$ is bounded on $\R_+$. Hence 
\begin{equation*}
    |x-\gamma|\lesssim_{\alpha, c, R, \delta} e^{-\delta (|x|/c)^{1/\alpha}}
\end{equation*}
for all $x$ such that $|\Phi_{\alpha,c}(x)|>\max\hch{2R+2,R+e^\alpha}$. Increasing the implicit constant yields the global bound.
\end{proof}

With Proposition~\ref{prop:zeros-decay} established, we can now state the next result, which summarizes the key fact that zeros of a Hölder function imply decay estimates.
\begin{proposition}\label{prop:Hol-decay}
Let $\alpha,c,\sigma>0$ and $\Gamma\subset \R^d$ be discrete. Suppose that $\D^-_{\Phi_{\alpha,c}}(\Gamma)>0$ and that $h:\R^d\to \C$ satisfies the condition
\begin{equation*}
    |h(y)-h(x)| \lesssim |x-y|^\sigma \quad \text{for every }  x,y\in \R^d.
\end{equation*}
If $h$ vanishes on $\Gamma$, then $|h(x)|\lesssim \exp \hp{-\frac{\sigma}{2}\hp{|x|/c}^{1/\alpha}}$ and the implicit constant is independent of $x$.
\end{proposition}
\begin{proof}
Let $x\in \R^d$. Using Proposition~\ref{prop:zeros-decay}, for $\delta=\frac{1}{2}$, we can find $\gamma \in \Gamma$ such that
\begin{equation*}
    |h(x)|=|h(x)-h(\gamma)| \lesssim |x-\gamma|^\sigma \lesssim \exp \hp{-\tfrac{\sigma}{2}\hp{|x|/c}^{1/\alpha}}.
\end{equation*}
\end{proof}

\begin{remark}\label{rem:cond-lim-der}
Let $\Gamma=\{\gamma_j\}_{j\in\Z}\subset \R$ be increasing, indexed so that $\gamma_j<\gamma_{j+1}$ for every $j\in\Z$, and such that $\gamma_j\to\pm\infty$ as $j\to\pm\infty$. Let $F$ be one of the functions in \eqref{eq:def-G-Phi} in dimension $d=1$. Then $\D_F^-(\Gamma)>0$ if and only if
\begin{equation}\label{eq:cond-lim-der}
    \limsup_{j\to\pm\infty} F'(\gamma_j)(\gamma_{j+1}-\gamma_j)<\infty.
\end{equation}
Moreover, the stronger condition $ \liminf_{j\to\pm\infty} F'(\gamma_j)(\gamma_{j+1}-\gamma_j)>0$ implies $\D_F^+(\Gamma)<\infty$, but the converse fails in general. In particular, for $F=G_\alpha$, the condition \eqref{eq:cond-lim-der} plays an analogous role in this work, though with a different meaning, as the supercriticality condition studied in \cite{KulikovNazarovSodin2025}.
\end{remark}

In dimension $d=1$, a discrete set $\Gamma$ with $\D^-_{G_\alpha}(\Gamma)>0$ satisfies analogous properties, but with polynomial rather than exponential-power decay and additional regularity assumptions are required. For $\Gamma=\mathbb Z_\alpha$ and Schwartz functions, the result below was proved in \cite[Lemma~4.1]{RamosSousa2022}. Since the same argument applies to sequences satisfying \eqref{eq:cond-lim-der}, we state it in this generality\footnote{Although $(-\Delta)^s f$ need not belong to $\mathcal{S}(\R)$, \eqref{eq:def-frac-Lapl} and the regularity in \eqref{eq:rang-FL-S} are sufficient for the argument.} and omit the proof.

\begin{proposition}\label{prop:JM-decay}
Let $f \in \mathcal{S}(\mathbb{R})$, $0<\alpha,s<1$, and $\Gamma\subset\R$ be discrete with $\D^-_{G_\alpha}(\Gamma)>0$. If $(-\Delta)^s f$ vanishes on $\Gamma$, then, for every $j \in \N$, there exists a constant $R_j>0$ such that for $|x|\ge R_j$
\begin{equation*}
    |(-\Delta)^{s}f(x)|\lesssim \hp{\int_\R |\widehat{f}(\xi)||\xi|^{j+2s}\dxi} |x|^{ j\hp{1-\frac{1}{\alpha}}}.
\end{equation*}
The implicit constant is independent of $x$.
\end{proposition}

\subsection[Proofs of Theorem~UP and Corollary~Hilbert Transform (i)]{Proofs of Theorem~\ref{thm:UP} and Corollary~\ref{cor:HT}~(i)}\label{sec:3.2}

The proof relies on standard tools from complex analysis. To use such machinery, we begin with two results on smoothness.
\begin{lemma}\label{lem:non-smooth-one-dim}
Let $I\subset\R$ be an open interval, $f \in C^\infty(I)$ and $h \in C(I)$ such that $g = f h \in C^\infty(I)$. For any $a \in I$, if $h$ is not smooth at $a$, then $f^{(k)}(a) = g^{(k)}(a)=0$ for all $k \in \mathbb{Z}_+$.
\end{lemma}

The lemma above is a straightforward consequence of standard calculus techniques, and its proof is therefore omitted.
\begin{lemma}\label{lem:non-smooth-high-dim}
Let $0<s<1$ and $U\subset\R^d$ be a neighborhood of the origin. Suppose that $g\in C^\infty(U)$ and that $h(\xi)= |\xi|^{2s}g(\xi)$ satisfies $h\in C^\infty(U)$. Then $\partial^\nu g(0)=0$ for every $\nu\in\Z_+^d$.
\end{lemma}
\begin{proof}
Let $r>0$ such that $B(0,r)\subset U$. Fix $\omega \in \mathbb{S}^{d-1}$ and set $g_\omega \in C^\infty(-r,r)$ by $g_\omega(t)=g(t\omega)$. Then $h_\omega(t)=|t|^{2s}g_\omega(t)$ belongs to $C^\infty(-r,r)$, and hence, by Lemma~\ref{lem:non-smooth-one-dim}, $g_\omega^{(k)}(0)=0$ for every $k\in\Z_+$. On the other hand, the chain rule gives
\begin{equation*}
    g_\omega^{(k)}(0) = \sum_{|\nu|=k}\frac{k!}{\nu!}\omega^\nu\partial^\nu g(0).
\end{equation*}
Thus, for each $k\in\Z_+$, $P_k(x) = \sum_{|\nu|=k} \frac{k!}{\nu!}x^\nu \partial^\nu g(0)$ is a homogeneous polynomial of degree $k$ which vanishes on $\mathbb{S}^{d-1}$. By homogeneity, it vanishes identically on $\R^d$, and therefore all its coefficients vanish. Hence $\partial^\nu g(0)=0$ for every $|\nu|=k$.
\end{proof}

\begin{proof}[Proof of Theorem~\ref{thm:UP}]
Fix $d\geq1$, $0<\alpha,s,\sigma<1$, and set $\beta=\alpha^{-1}$.

\textit{Item~{\hyperref[item:1-thmUP]{$(i)$}}}. Assume $f \in X_{s,\sigma}(\mathbb{R}^d)$ vanishes on $\Lambda$ and $(-\Delta)^s f$ vanishes on $M$, with $\D^-_{\Phi_{1,c_1}}(\Lambda) > 0$ and $\D^-_{\Phi_{\beta,c_2}}(M) > 0$ for some $c_1,c_2>0$. For $f \in X_{s,\sigma}(\mathbb{R}^d)$, note that for every $x,y\in \R^d$
\begin{equation}\label{eq:decay-f}
    |f(x)-f(y)|\lesssim |x-y|^{\theta},
\end{equation}
where $\theta = \min\hch{2s+\sigma, 1}$. This is clear for $0<s<\frac{1}{2}$, and for $\frac{1}{2}\leq s<1$
\begin{equation*}
    |f(x)-f(y)| \leq \int_0^1 |\nabla f (x+t(y-x))\cdot (y-x)|\dt \leq \hn{\nabla f}_\infty |x-y|.
\end{equation*}
Using Proposition~\ref{prop:Hol-decay} with \eqref{eq:decay-f} and the set of zeros $\Lambda$, we have
\begin{equation}\label{eq:f-exp-dec}
    |f(x)|\lesssim e^{-\frac{\theta}{2c_1}|x|}. 
\end{equation}
Furthermore, \eqref{eq:f-exp-dec} implies that $\widehat{f}$ has an analytic continuation in the tube $\Omega = \hch{|\IM(z)| <\tfrac{\theta}{4\pi c_1}}$. Using the regularity estimate \eqref{eq:Silv-regul}, the assumption on $M$ and Proposition~\ref{prop:Hol-decay} yields
\begin{equation}\label{eq:frac-L-subexp}
    |(-\Delta)^s f(y)|\lesssim e^{-b |y|^\alpha}
\end{equation}
for some $b>0$ and all $y\in \R^d$. In particular, $y^\nu (-\Delta)^s f(y) \in L^1(\R^d)$ for every $\nu\in\Z_+^d$, and $\mathcal{F}[(-\Delta)^s f] \in C^\infty(\R^d)$. Let $\varphi(x) = e^{-\pi |x|^2}$ and define $\psi = f * \varphi$. It is known that $\varphi \in \mathcal{S}(\R^d)$, $\widehat{\varphi} = \varphi$, and $(-\Delta)^s\varphi \in \L_s(\R^d)$ by \eqref{eq:rang-FL-S}. The pointwise singular-integral definition of $(-\Delta)^s f$ for $f\in X_{s,\sigma}(\R^d)$, Fubini's theorem, together with \eqref{eq:f-exp-dec} and \eqref{eq:frac-L-subexp}, yields for any $x\in \R^d$
\begin{equation*}
    [(-\Delta)^sf* \varphi] (x) = (-\Delta)^s \psi (x)=[f*(-\Delta)^s \varphi](x).
\end{equation*}
Taking the Fourier transform on both sides, and using \eqref{eq:frac-L-subexp} and \eqref{eq:f-exp-dec} again, we obtain for all $\xi \in \R^d$
\begin{equation*}
    \mathcal{F}[(-\Delta)^sf](\xi) \widehat{\varphi} (\xi)= \widehat{f} (\xi) \mathcal{F}[(-\Delta)^s\varphi](\xi) = |\xi|^{2s}\widehat{f} (\xi) \widehat{\varphi} (\xi).
\end{equation*}
Since $\widehat\varphi$ does not vanish on $\R^d$, it follows that $\mathcal{F}[(-\Delta)^sf](\xi) = |\xi|^{2s}\widehat f(\xi)$ pointwise in $\xi\in\R^d$. Therefore Lemma~\ref{lem:non-smooth-high-dim}, applied with $g=\widehat f$, gives $\partial^\nu \widehat f(0)=0$ for every $\nu\in\Z_+^d$. Let $F$ denote the holomorphic extension of $\widehat{f}$ to $\Omega$. Choose $r>0$ such that the polydisc
\begin{equation*}
    P(0,r) = \hch{z\in\C^d \st |z_j|<r,\ j=1,\dots,d}
\end{equation*}
satisfies $\overline{P(0,r)} \subset \Omega$. Since $F$ is holomorphic in $P(0,r)$, it has a convergent Taylor expansion
\begin{equation*}
    F(z) = \sum_{\nu\in\Z_+^d} \frac{\partial_z^\nu F(0)}{\nu!}z^\nu \quad \text{for } z \in P(0,r).
\end{equation*}
By the Cauchy--Riemann equations, the complex derivatives of $F$ at the origin agree with the corresponding real derivatives of its restriction to $\R^d$. Hence, for every multiindex
\begin{equation*}
    \partial_z^\nu F(0) = \partial^\nu \widehat f(0) = 0,
\end{equation*}
and all Taylor coefficients of $F$ at the origin vanish. Thus $F\equiv 0$ in $P(0,r)$, and because $\Omega$ is connected, the identity theorem for holomorphic functions of several variables implies that $F\equiv 0$ on $\Omega$. Then $\widehat{f}\equiv0$, and by Fourier inversion $f\equiv 0$.

\textit{Item~\hyperref[item:2-thmUP]{$(ii)$}}.
By the proof of Theorem~\ref{thm:UP}~\hyperref[item:1-thmUP]{$(i)$}, the zeros of $f$ force that $\widehat{f}$ is an analytic function in the strip $\Omega = \hch{\hmm{\IM(z)}<\tfrac{1}{4\pi c}}.$ Let $k \in \Z_+$. By Proposition~\ref{prop:JM-decay}, for every $j\in \N$, there exists $R_j>0$ such that
\begin{equation*}
    |(-\Delta)^{s}f(x)| \lesssim |x|^{j\hp{1-\frac{1}{\alpha}}}
\end{equation*}
for $|x| \ge R_j$. Choosing $j>\lceil \alpha(k+1)(1-\alpha)^{-1}\rceil$, we see that $|x|^k(-\Delta)^{s}f(x)\in L^1(\R)$ and it follows that $\mathcal{F}((-\Delta)^sf)\in C^\infty(\R)$. Using this, we have $\mathcal{F}((-\Delta)^sf)(\xi) = |\xi|^{2s}\widehat{f}(\xi)$ pointwise and since $|\xi|^{2s}$ is not smooth at $0$, it follows by Lemma~\ref{lem:non-smooth-one-dim} that $\widehat{f}^{(m)}(0)=0$ for every $m\in \Z_+$. By analytic continuation of $\widehat{f}$, which has a zero of infinite order, we conclude that $\widehat{f}\equiv 0$ in $\Omega$, and the result follows by Fourier inversion.
\end{proof}

Once we have finished the proof of Theorem~\ref{thm:UP}, we can conclude the result on uniqueness pairs for the Hilbert transform.

\begin{proof}[Proof of Corollary~\ref{cor:HT}~{\hyperref[item:1-HT]{$(i)$}}]
Let $\Lambda, M$ discrete, with $\D^-_{\Phi_{1,c}}(\Lambda)$ and $\D^-_{G_{\alpha}}(M)$ positive for some $c>0$ and $0<\alpha<1$. Assume that $f\in \Sp(\R)$ vanishes on $\Lambda$ and ${\bf H}f$ vanishes on $M$. By the real linearity of $\mathbf H$, if $f=u+iv$, then $u$ and $v$ satisfy the same hypotheses as $f$. Hence, the complex-valued case follows from the real-valued one, and we may assume that $f$ is real-valued. In $L^2(\R)$, 
\begin{equation*} 
    2\pi |\xi|\widehat{f}(\xi)= (2\pi i \xi)(-i\sgn(\xi))\widehat{f}(\xi) = (2\pi i \xi)\widehat{{\bf H}f}(\xi). 
\end{equation*} 
Thus $2\pi(-\Delta)^{\frac{1}{2}} f = ({\bf H}f)'$ almost everywhere on $\mathbb{R}$. Since both sides are smooth, the equality holds pointwise. Moreover, it follows from the assumption $\D^-_{G_{\alpha}}(M) > 0$ and the mean value theorem that the zero set of $({\bf H}f)'$ contains a discrete subset $M'$ satisfying $\D^-_{G_{\alpha}}(M') > 0$. By Theorem~\ref{thm:UP}~\hyperref[item:2-thmUP]{$(ii)$} with $s=\frac{1}{2}$, $f$ vanishes identically. 
\end{proof}

\subsection{Frame bounds}\label{sec:3.3}
Let $0<s<1$. Since Theorem~\ref{thm:UP} provides a uniqueness result, it is natural to ask whether this uniqueness can be made stable in a Hilbert space $\Hm\subset X_{s,\sigma}(\R^d)$.

Let $Y=\ell^2(\Lambda,w_\Lambda)\oplus \ell^2(M,w_M)$, where $w_\Lambda,w_M$ are positive weights. We consider the sampling map $T:\Hm\to Y$ by
\begin{equation*}
    Tf = \hp{\hch{f(\lambda)}_{\lambda\in\Lambda}, \hch{(-\Delta)^s f(\mu)}_{\mu\in M}},
\end{equation*}
whenever the right-hand side belongs to $Y$. The question is whether $T$ satisfies the frame inequalities
\begin{equation}\label{eq:frame-bds}
    \|f\|_{\Hm}^2 \lesssim \sum_{\l\in\Lambda}w_\Lambda(\lambda)\hmm{f(\lambda)}^2 + \sum_{\mu\in M}w_M(\mu)\hmm{(-\Delta)^s f(\mu)}^2 \lesssim \|f\|_{\Hm}^2.
\end{equation}
In such a case, $T$ would be bounded below and have a closed range. Since $Y$ is a Hilbert space, $T(\Hm)$ would be complemented in $Y$, and therefore $T$ would admit a bounded left inverse $L: Y\to\Hm$ with $LT=I_{\Hm}$. In particular, the frame bounds imply that $f\in \Hm$ can be reconstructed, in a stable way, from the values $\hch{f(\lambda)}_{\l\in \Lambda}$ and $\hch{(-\Delta)^s f(\mu)}_{\mu \in M}$ by the interpolation formula
\begin{equation*}
    f = \sum_{\lambda\in\Lambda} f(\lambda)a_\lambda +\sum_{\mu\in M} (-\Delta)^s f(\mu)b_\mu
\end{equation*}
with convergence in $\Hm$, see \cite[Theorem~1.8, Corollary~2.2]{KulikovNazarovSodin2025}.

We now show that such two-sided frame bounds cannot hold in the present setting under natural assumptions on $\Hm$. Assume that $\Hm$ contains $C_c^\infty(\mathbb R^d)$ and that $\|h\|_{L^2(\mathbb R^d)} \lesssim \|h\|_{\Hm}$ for $h\in C_c^\infty(\R^d)$. Suppose, moreover, that the upper estimate in \eqref{eq:frame-bds} holds. Since $\Lambda\cup M$ is discrete, we may choose $x_0\in\R^d$ and $r_0>0$ such that $B(x_0,2r_0)\cap(\Lambda\cup M)=\emptyset$. Choose a $0\neq \varphi\in C_c^\infty(B(0,1))$, with $\varphi\geq 0$, and for $0<r\leq r_0$ set
\begin{equation*}
    g_r(x)=\varphi\hp{\frac{x-x_0}{r}}, \qquad f_r=\frac{g_r}{\|g_r\|_{\Hm}}.
\end{equation*}
Then $f_r|_\Lambda=0$. Also, for $\mu\in M$, the point $\mu$ is separated from $\supp g_r$, and the formula \eqref{eq:def-FL-pv} gives
\begin{equation*}
    (-\Delta)^s g_r(\mu) = -c_{d,s}\int_{\R^d} \frac{g_r(y)}{|\mu-y|^{d+2s}}\dy.
\end{equation*}
Using $|\mu-y|\asymp 1+|\mu|$ for $y\in B(x_0,r)$ and $\mu \in M$, we obtain 
\begin{equation}\label{eq:bound-frac-gr}
    |(-\Delta)^s g_r(\mu)| \asymp r^{d}(1+|\mu|)^{-d-2s}.
\end{equation}
Using this and applying the upper frame bound in \eqref{eq:frame-bds} for $g_{r_0}\in \Hm$, and using that $g_{r_0}|_\Lambda=0$, we get $\sum_{\mu\in M} w_M(\mu)(1+|\mu|)^{-2d-4s}<\infty$. On the other hand, $\|g_r\|_{\Hm}\gtrsim \|g_r\|_{L^2}\asymp r^{d/2}$ and \eqref{eq:bound-frac-gr} yields
\begin{equation*}
    \sum_{\mu\in M} w_M(\mu)|(-\Delta)^s f_r(\mu)|^2 \lesssim \sum_{\mu\in M} w_M(\mu) \hp{r^{d/2}(1+|\mu|)^{-d-2s}}^2\lesssim r^d,
\end{equation*}
by the definition of $f_r$. Therefore $\|Tf_r\|_Y\to0$ as $r\to0^+$, while $\|f_r\|_{\Hm}=1$. Consequently, whenever the upper estimate holds in such a Hilbert space $\Hm$, the lower bound in \eqref{eq:frame-bds} cannot hold.

This shows where the analogy with the Fourier transform argument fails. If $g_r$ is the bump function used above, localized at scale $r$, then $\widehat{g_r}$ spreads over frequencies of scale $r^{-1}$, and the uncertainty principle prevents the localization argument. In the present problem, however, both sets of zeros $\Lambda$ and $M$ are imposed on the physical side, so such localization is available.

\begin{remark}
One may ask whether the obstruction disappears when the Gagliardo seminorm replaces the Hilbert norm. Let $0<\eta<1$ and consider the Gagliardo seminorm
\begin{equation*}
    \hbra{h}_{\dot H^\eta}^2 \defi \int_{\R^d}\int_{\R^d} \frac{\hmm{h(x)-h(y)}^2}{\hmm{x-y}^{d+2\eta}}\dx\dy.
\end{equation*}
Assume first that the analogue of the upper estimate holds. Using the same localized functions $g_r$, one has \eqref{eq:bound-frac-gr} and $\hbra{g_r}_{\dot H^\eta}^2 = r^{d-2\eta} \hbra{\varphi}_{\dot H^\eta}^2$. Hence, if $f_r=\frac{g_r}{\hbra{g_r}_{\dot H^\eta}}$, then $f_r|_\Lambda=0$ and $\hbra{f_r}_{\dot H^\eta}=1$. Therefore, arguing as above, the analogue of the lower estimate in \eqref{eq:frame-bds} would imply
\begin{equation*}
    1 \lesssim \sum_{\mu\in M}w_M(\mu)\hmm{(-\Delta)^s f_r(\mu)}^2 \lesssim \frac{r^{2d}}{\hbra{g_r}_{\dot H^\eta}^2} \sum_{\mu\in M}w_M(\mu)(1+\hmm{\mu})^{-2d-4s} \lesssim r^{d+2\eta},
\end{equation*}
which is impossible as $r\to0^+$.
\end{remark}

\section{Non-uniqueness pairs}\label{sec:4}

The purpose of this section is to prove Theorem~\ref{thm:nup-strong}. We shall, in fact, prove a more flexible interpolation result from which the previous theorem will follow, see Theorem~\ref{thm:interpol-NUP}. We first construct Gevrey-class functions with prescribed values for the function and its fractional Laplacian at the origin, together with subexponential decay estimates. After translations and dilations, these functions are used to define a synthesis operator on the weighted sequence space associated with the sets $\Lambda$ and $M$. Composing this synthesis operator with the trace map produces an operator of the form $I+E$. The key point is to choose the scales so that the off-diagonal operator $E$ has norm strictly smaller than one. Once this is done, the interpolation operator is obtained using the inverse of $I+E$. This inversion of $I+E$ follows the strategy employed in the non-uniqueness results of \cite{Adve2023} and \cite{KulikovNazarovSodin2025}.

\subsection{Gevrey class}\label{sec:4.1}
For the discussion in this subsection, see \cite[Sections~1.4--1.6]{Rodino1993}. Let $r>1$ and $U\subset\R^d$ be open. We denote by $G^r(U)$ the Gevrey class of order $r$, that is, the space of all $h\in C^\infty(U)$ such that, for every compact set $K\Subset U$, there exist constants $C_K,A_K>0$ satisfying
\begin{equation*}
    |\partial^\nu h(\xi)| \leq C_K (A_K)^{|\nu|}(\nu!)^r
\end{equation*}
for $\xi\in K$ and $\nu\in\Z_+^d$. We also write $G_c^r(U)=G^r(U)\cap C_c^\infty(U)$. Since $r>1$, $G_c^r(U)$ contains nonnegative functions supported in arbitrary balls compactly contained in $U$.

\begin{remark}\label{rem:PW-Gevrey}
Let $r>1$ and let $h\in G_c^r(\R^d)$. Set $\delta=1/r$. Then, for every multiindex $\nu \in \Z_+^d$, there are constants $C_\nu,c_\nu>0$ such that 
\begin{equation*}
    |\partial^\nu\mathcal{F}^{-1} h(x)| \leq C_\nu e^{-c_\nu |x|^\delta}
\end{equation*}
for $x\in\R^d$. If $m:\R^d\to \C$ is real analytic in a neighbourhood of $\supp h$, then the same inequality holds with $h$ replaced by $mh$.
\end{remark}
We now construct the functions that will be used in the proof of Theorem~\ref{thm:interpol-NUP}.
\begin{proposition}\label{prop:phi-psi}
Let $s>0$, $U=\hch{\xi\in\R^d \st 1<|\xi|<2}$ and $0<\delta<1$. Then, there exist functions $\varphi,\psi\in \Sp(\R^d)$ such that $\supp \widehat\varphi\cup\supp\widehat\psi \Subset U$,
\begin{equation}\label{eq:phi-psi-point}
    \varphi(0)=1, \quad (-\Delta)^s\varphi(0)=0, \quad \psi(0)=0, \quad (-\Delta)^s\psi(0)=1,
\end{equation}
and, for every multiindex $\nu \in \Z_+^d$, there exist constants $C_\nu,c_\nu>0$ such that
\begin{equation*}
    |\partial^\nu\varphi(x)|+|\partial^\nu\psi(x)|+ |\partial^\nu (-\Delta)^s\varphi(x)|+|\partial^\nu (-\Delta)^s\psi(x)| \leq C_\nu e^{-c_\nu |x|^\delta}.
\end{equation*}
\end{proposition}

\begin{proof}
Let $m_s(\xi)=|\xi|^{2s}$ and set $r=1/\delta$. Since $m_s$ is real analytic on $U$, multiplication by $m_s$ preserves $G_c^r(U)$. Consider the linear map $T:G_c^r(U)\to\C^2$ defined as
\begin{equation*}
    T(h)=\hp{\int_{\R^d} h(\xi)\dxi,\int_{\R^d} m_s(\xi)h(\xi)\dxi}.
\end{equation*}
We will prove that $T$ has rank two. Indeed, otherwise there exist $(z_1,z_2)\in\C^2\setminus\hch{(0,0)}$ such that
\begin{equation*}
    \int_{\R^d}\hp{z_1+z_2m_s(\xi)}h(\xi)\dxi=0
\end{equation*}
for every $h\in G_c^r(U)$. Since $G_c^r(U)$ contains nonnegative functions supported in arbitrarily small balls compactly contained in $U$, this implies that $z_1+z_2m_s(\xi)=0$ for $\xi\in U$, which is impossible unless $(z_1,z_2)=(0,0)$. Choose $h_1,h_2\in G_c^r(U)$ such that $T(h_1),T(h_2)$ form a basis of $\C^2$, and set $u_j=\mathcal F^{-1}h_j$. Then
\begin{equation*}
    M=\begin{pmatrix}
        u_1(0) & u_2(0) \\
        (-\Delta)^su_1(0) & (-\Delta)^su_2(0)
    \end{pmatrix}
\end{equation*}
is invertible. Let $v^{(j)}\in\C^2$ be defined by $Mv^{(j)}=e_j$, $j=1,2$. For $v^{(j)}=(v^{(j)}_1,v^{(j)}_2)$, define
\begin{equation*}
    \varphi ={v^{(1)}_1u_1+v^{(1)}_2u_2}, \qquad \psi ={v^{(2)}_1u_1+v^{(2)}_2u_2}.
\end{equation*}
Then \eqref{eq:phi-psi-point} holds by definition, while $\supp \widehat\varphi\cup\supp\widehat\psi \Subset U$ follows from $\widehat\varphi,\widehat\psi\in\operatorname{span}\hch{h_1,h_2}$. Finally, the subexponential decay follows from Remark~\ref{rem:PW-Gevrey}, applied to both $h_j$ and $m_sh_j$.
\end{proof}

\subsection{\texorpdfstring{$\bf{\rho}$}{ρ}-separated sets}\label{sec:4.2}

Recall that a discrete set $\Gamma\subset\R^d$ is called $\rho$-separated if there exist $A,B>0$ such that for $\gamma\in\Gamma$
\begin{equation}\label{eq:const-def-sep}
    \dist(\gamma,\Gamma\setminus\{\gamma\}) \ge A e^{-B|\gamma|^\rho}.
\end{equation}

For $\alpha>1$, the $\frac1\alpha$-separation condition is a local separation assumption at the exponential scale, whereas $\D_{\Phi_{\alpha,c}}^+(\Gamma)<\infty$ is only a global density condition. Nevertheless, both assumptions imply the following counting estimate. For simplicity, we prove it only for $\rho$-separated sets.

\begin{proposition}\label{prop:dens-cntg}
Let $0<\rho<1$ and let $\Gamma\subset\R^d$ be $\rho$-separated. Then there exists a constant $C>0$ such that for $R\geq 1$
\begin{equation*}
    \#\hch{\gamma\in\Gamma \st |\gamma|\leq R} \lesssim e^{CR^\rho}.
\end{equation*}
\end{proposition}
\begin{proof}
Let $A,B>0$ be the constants of \eqref{eq:const-def-sep} and set $r_R=\frac{A}{4}e^{-B R^\rho}$. For $\gamma\in\Gamma\cap B(0,R)$ the balls $B(\gamma,r_R)$ are pairwise disjoint and contained in $B(0,R+r_R)$. Hence, writing $N(R)=\#(\Gamma\cap B(0,R))$ and comparing volumes $N(R)|B(0,r_R)|\leq |B(0,R+r_R)|$. Using $r_R\lesssim 1$, we have $N(R)\lesssim R^d e^{Bd R^\rho}$, and the polynomial factor is absorbed into $e^{CR^\rho}$.
\end{proof}

We shall also use the following elementary enlargement property of $\rho$-separated sets.
\begin{proposition}\label{prop:Y-ela-two-sets}
Let $0<\rho<1$, and let $\Lambda,M\subset\R^d$ be $\rho$-separated. Then there is an infinite set $Y$, with $Y\subset\R^d\setminus(\Lambda\cup M)$, such that $\Lambda\cup Y$ is $\rho$-separated.
\end{proposition}
\begin{proof}
Let $A, B>0$ be separation constants for $\Lambda$. Pick a sequence $R_j\to\infty$ such that $R_{j+1}>4R_j$ and $R_1\geq 1$ is large. By Proposition~\ref{prop:dens-cntg}, there is $C_0>0$ such that
\begin{equation*}
    \#\hch{\lambda\in\Lambda\st |\lambda|\le 3R_j} + \#\hch{\mu\in M\st |\mu|\le 3R_j} \lesssim e^{C_0R_j^\rho}.
\end{equation*}
Choose $C>C_0+2$. Then
\begin{equation*}
    \hmm{\bigcup_{\gamma\in(\Lambda\cup M)\cap B(0,3R_j)} B(\gamma,e^{-C R_j^\rho})} \lesssim e^{C_0R_j^\rho} e^{-dCR_j^\rho}.
\end{equation*}
This tends to $0$ as $j\to\infty$. On the other hand, consider $O_j=\hch{x\in\R^d \st R_j<|x|<2R_j}$ for $j\ge 1$. We have $|O_j|\asymp R_j^d$, and hence, for $j$ sufficiently large, the above union of balls cannot cover $O_j$. Therefore, for such $j$ one may choose
\begin{equation*}
    y_j\in O_j\setminus \bigcup_{\gamma\in(\Lambda\cup M)\cap B(0,3R_j)} B(\gamma,e^{-C R_j^\rho}).
\end{equation*}
If $\gamma\notin B(0,3R_j)$ and $y_j\in O_j$, then $|y_j-\gamma|\ge R_j$, which is larger than $e^{-CR_j^\rho}$ for large $j$. Then one can choose $y_j\in O_j$ such that
\begin{equation}\label{eq:Y-avoid-two}
    \dist(y_j,\Lambda\cup M)\ge e^{-C R_j^\rho}
\end{equation}
for some constant $C>0$ independent of $j$. Discarding finitely many initial annuli and reindexing, set $Y=\hch{y_j\st j\ge1}$. We now prove that $\widetilde\Lambda=\Lambda\cup Y$ is $\rho$-separated. First, consider $y_j\in Y$. Since $R_j<|y_j|<2R_j$, \eqref{eq:Y-avoid-two} implies $\dist(y_j,\Lambda)\ge e^{-C'|y_j|^\rho}.$ Moreover, the condition $R_{j+1}>4R_j$ implies that the annuli are separated, hence there exist $A_1,B_1>0$ such that
\begin{equation}\label{eq:adm1-two}
    \dist(y_j,\widetilde\Lambda\setminus\hch{y_j}) \ge A_1 e^{-B_1|y_j|^\rho}.
\end{equation}
It remains to check the $\rho$-separability for other points, namely for $\lambda\in\Lambda$. Fix $\lambda\in\Lambda$ and $j\ge1$. If $R_j\le |\lambda|/3$, then
\begin{equation*}
    |\lambda-y_j| \ge |\lambda|-|y_j| >|\lambda|-2R_j \ge \frac{|\lambda|}{3}\geq R_1.
\end{equation*}
If $R_j>2|\lambda|$, then
\begin{equation*}
    |\lambda-y_j| \ge |y_j|-|\lambda| >R_j-|\lambda| >\frac{R_j}{2} \ge \frac{R_1}{2}.
\end{equation*}
Finally, if $\frac{|\lambda|}{3}<R_j\le 2|\lambda|$, then $|\lambda|<3R_j$, so $\lambda\in\Lambda\cap B(0,3R_j)$, and by the choice of $y_j$,
\begin{equation*}
    |\lambda-y_j|\ge e^{-C R_j^\rho} \ge e^{-C2^\rho|\lambda|^\rho}.
\end{equation*}
Thus, possibly after changing the constant $A_1$ and $B_1$, $|\lambda-y_j|\ge A_1 e^{-B_1|\lambda|^\rho}$ for every $j\ge1$. Taking the infimum over $j$ gives $\dist(\lambda,Y)\ge A_1 e^{-B_1|\lambda|^\rho}.$ Combining this with the original $\rho$-separation of $\Lambda$, we obtain constants $A_2,B_2>0$ such that
\begin{equation}\label{eq:adm2-two}
    \dist(\lambda,\widetilde\Lambda\setminus\hch{\lambda}) \ge A_2 e^{-B_2|\lambda|^\rho}.  
\end{equation}
The bounds \eqref{eq:adm1-two} and \eqref{eq:adm2-two} imply the result.
\end{proof}

Finally, we prove a summability property, which follows from the choice of two distinct scales.
\begin{proposition}\label{prop:cntg-sum}
Let $0<\rho<1$ and let $\Gamma\subset\R^d$ be $\rho$-separated. If $\eta>\rho$ and $\tau>0$, then
\begin{equation*}
    \sum_{\gamma \in\Gamma} e^{-\tau |\gamma|^\eta}<\infty.
\end{equation*}
\end{proposition}
\begin{proof}
First, the set $\Gamma\cap B(0,1)$ is finite. For $j\geq0$ define $O_j=\hch{\gamma \in\Gamma \st 2^j\leq |\gamma|<2^{j+1}}$. By Proposition~\ref{prop:dens-cntg} we have $\#O_j\lesssim e^{C2^{j\rho}}$. Hence $\sum_{\gamma\in O_j} e^{-\tau |\gamma|^\eta} \lesssim\exp\hp{C2^{j\rho}-\tau 2^{j\eta}}$. Since $\eta>\rho$, the exponent is at most $-\frac{\tau}{2}2^{j\eta}$ for all large $j$. Summing over $j$ gives the proposition.
\end{proof}

\subsection{An interpolation result}\label{sec:4.3}

In what follows, we assume that $0<\rho<\eta<1$. We first introduce a weighted sequence space.
\begin{definition}
Let $\Lambda, M\subset\R^d$ be discrete. For $\tau,L>0$, define
\begin{equation*}
    X_{\tau,L}(\Lambda, M) = \hch{(\alpha,\beta)\st  \hn{(\alpha,\beta)}_{X_{\tau,L}} \defi \sup_{\l\in\Lambda}e^{\tau|\l|^\eta+L|\l|^\rho}|\alpha(\l)| + \sup_{\mu\in M}e^{\tau|\mu|^\eta}|\beta(\mu)| <\infty},
\end{equation*}
where $\alpha:\Lambda\to\C$ and $\beta:M\to\C$.
\end{definition}
Equipped with this norm, $X_{\tau, L}(\Lambda, M)$ is a Banach space, since it is the product of two weighted $\ell^\infty$ spaces. We now state the interpolation result that will be used to prove Theorem~\ref{thm:nup-strong}.
\begin{theorem}\label{thm:interpol-NUP}
Let $0<\rho<\eta<1$, $s>0$ and $\Lambda,M \subset\R^d$ be $\rho$-separated. For every $\tau>0$, there exist $L>0$ and a continuous linear map $Q_{\tau,L}:X_{\tau,L}(\Lambda,M)\to\Sp(\R^d)$ such that, if $f=Q_{\tau,L}(\alpha,\beta)$, then $(-\Delta)^sf$ belongs to $\Sp(\R^d)$ and
\begin{equation}\label{eq:two-set-interpolation}
    f(\l)=\alpha(\l)\text{ for }\l\in\Lambda, \qquad (-\Delta)^sf(\mu)=\beta(\mu) \text{ for } \mu\in M.
\end{equation}
Moreover, for every multiindex $\nu \in \Z_+^d$, there exist constants $C_\nu,c_\nu>0$ such that
\begin{equation}\label{eq:interp-decay}
    |\partial^\nu f(x)|+|\partial^\nu(-\Delta)^sf(x)| \leq C_\nu\hn{(\alpha,\beta)}_{X_{\tau,L}}e^{-c_\nu|x|^\eta}.
\end{equation}
\end{theorem}

The sets $\Lambda$ and $M$ in the theorem above are not required to be disjoint. Before turning to the proof, we record two auxiliary lemmas. The first is a summation estimate for the exponentially small radii that will appear in the construction. It will be used to prove the convergence of the relevant series and to obtain the desired subexponential decay of \eqref{eq:interp-decay}.

\begin{lemma}\label{lem:properties_rl}
Let $0<\rho<\eta<\delta<1$ and $\Gamma\subset\R^d$ be $\rho$-separated. Suppose that $0<r_\gamma\leq1$, $\gamma\in\Gamma$, is a family of radii such that, outside a finite subset of $\Gamma$, $r_\gamma=e^{-\kappa|\gamma|^\rho}$ for some $\kappa>0$. Then the following assertions hold.
\begin{enumerate}
\item\label{item:absor-scale} For every $K\geq0$ and $\varepsilon>0$, there exists $C_{K,\varepsilon}>0$ such that $r_\gamma^{-K} \leq C_{K,\varepsilon}e^{\varepsilon|\gamma|^\eta}$ for $\gamma\in\Gamma$.
\item\label{item:decay-scale} For every $K\geq0$ and $\sigma, a>0$, there exists a constant $c>0$ such that
\begin{equation*}
    \sum_{\gamma\in\Gamma} e^{-\sigma|\gamma|^\eta} r_\gamma^{-K} \exp\hp{-a \hmm{\frac{x-\gamma}{r_\gamma}}^\delta} \lesssim e^{-c|x|^\eta} \quad \text{for } x\in\R^d.
\end{equation*}
\end{enumerate}
\end{lemma}
\begin{proof}
\textit{Item} \hyperref[item:absor-scale]{$(i)$}. If $\gamma$ lies outside of the finite exceptional set, then $r_\gamma^{-K}=e^{K\kappa|\gamma|^\rho}.$ Since $\eta>\rho$, for every $\varepsilon>0$ there is $C_{K,\varepsilon}>0$ such that $K\kappa|\gamma|^\rho\leq \varepsilon|\gamma|^\eta+C_{K,\varepsilon}$, for $\gamma\in\Gamma$. Increasing $C_{K,\varepsilon}$ absorbs the finitely many exceptional radii.

\textit{Item} \hyperref[item:decay-scale]{$(ii)$}. Choose $0<\varepsilon<\sigma$. By item~\hyperref[item:absor-scale]{$(i)$} we have $e^{-\sigma|\gamma|^\eta}r_\gamma^{-K} \leq C_{K,\varepsilon} e^{-(\sigma-\varepsilon)|\gamma|^\eta}$. Replacing $\sigma-\varepsilon$ by $\sigma_1$, it is enough to bound
\begin{equation*}
    S(x)= \sum_{\gamma\in\Gamma} e^{-\sigma_1|\gamma|^\eta} \exp\hp{-a\hmm{\frac{x-\gamma}{r_\gamma}}^\delta}
\end{equation*}
for $\sigma_1>0$. To prove that, we split the sum into $|\gamma|\geq |x|/2$ and $|\gamma|<|x|/2$. In the first region, $e^{-\sigma_1|\gamma|^\eta} \leq e^{-c|x|^\eta}e^{-\frac{\sigma_1}{2}|\gamma|^\eta}$ for a constant $c>0$. Hence Proposition~\ref{prop:cntg-sum} gives
\begin{equation*}
    \sum_{\substack{\gamma\in\Gamma\\ |\gamma|\geq |x|/2}} e^{-\sigma_1|\gamma|^\eta} \exp\hp{-a\hmm{\frac{x-\gamma}{r_\gamma}}^\delta} \lesssim e^{-c|x|^\eta}.
\end{equation*}
In the second region, $|x-\gamma|\geq |x|/2$ and $r_\gamma\leq1$, so
\begin{equation*}
    \exp\hp{-a\hmm{\frac{x-\gamma}{r_\gamma}}^\delta}\leq e^{-a2^{-\delta}|x|^\delta} \lesssim e^{-c|x|^\eta},
\end{equation*}
because $\delta>\eta$. A second application of Proposition~\ref{prop:cntg-sum} gives the result.
\end{proof}

The second lemma chooses the radii so that four sums are uniformly small. These estimates will later be used to control the error operator in the weighted sequence norm.
\begin{lemma}\label{lem:existence_rl}
Let $s>0$, $0<\rho<\eta<\delta<1$, and $\Lambda, M \subset\R^d$ both be $\rho$-separated. For $\tau, L>0$ set
\begin{equation*}
    \Omega_1(x)=\tau |x|^\eta+L|x|^\rho, \qquad \Omega_0(x)=\tau |x|^\eta.
\end{equation*}
Let $A_\Lambda,B_\Lambda>0$ and $A_M,B_M>0$ be the separating constants for $\Lambda$ and $M$, respectively. Fix $\kappa_0>B_\Lambda, \kappa_1>B_M$, and assume that
\begin{equation}\label{eq:simul-summation}
    \sum_{\lambda\in\Lambda} e^{-(L-2s\kappa_0)|\lambda|^\rho}<\infty, \qquad  \sum_{\mu\in M} e^{-(2s\kappa_1-L)|\mu|^\rho}<\infty.
\end{equation}
Then, for every $a, \varepsilon>0$, there exist finite sets $F_\Lambda\subset\Lambda$, $F_M\subset M$, and radii $0<r_\lambda,t_\mu\leq1$ satisfying
\begin{equation}\label{eq:radius-small}
    r_\lambda=e^{-\kappa_0|\lambda|^\rho} \text{ for } \lambda\in\Lambda\setminus F_\Lambda, \qquad  t_\mu=e^{-\kappa_1|\mu|^\rho} \text{ for } \mu\in M\setminus F_M,
\end{equation}
and such that
\begin{equation}\label{eq:simul-Lambda-Lambda}
    \sup_{\lambda'\in\Lambda} \sum_{\substack{\lambda\in\Lambda\\ \lambda\ne\lambda'}} e^{\Omega_1(\lambda')-\Omega_1(\lambda)} r_\lambda^{-2s} \exp\hp{-a\hmm{\frac{\lambda'-\lambda}{r_\lambda}}^\delta} <\varepsilon,
\end{equation}
\begin{equation}\label{eq:simul-M-M}
    \sup_{\mu'\in M} \sum_{\substack{\mu\in M\\ \mu\ne\mu'}} e^{\Omega_0(\mu')-\Omega_0(\mu)} \exp\hp{-a\hmm{\frac{\mu'-\mu}{t_\mu}}^\delta} <\varepsilon,
\end{equation}
\begin{equation}\label{eq:simul-Lambda-M}
    \sup_{\mu\in M} \sum_{\substack{\lambda\in\Lambda\\ \lambda\ne\mu}} e^{\Omega_0(\mu)-\Omega_1(\lambda)} r_\lambda^{-2s} \exp\hp{-a\hmm{\frac{\mu-\lambda}{r_\lambda}}^\delta} <\varepsilon,
\end{equation}
\begin{equation}\label{eq:simul-M-Lambda}
    \sup_{\lambda\in\Lambda} \sum_{\substack{\mu\in M\\ \mu\ne\lambda}} e^{\Omega_1(\lambda)-\Omega_0(\mu)} t_\mu^{2s} \exp\hp{-a\hmm{\frac{\lambda-\mu}{t_\mu}}^\delta} <\varepsilon.
\end{equation}
\end{lemma}

\begin{proof}
From Proposition~\ref{prop:dens-cntg}, for $\Gamma=\Lambda,M$, there exists a constant $C>0$ such that for $R\geq 1$
\begin{equation}\label{eq:counting-simul}
    \#\hch{\gamma\in\Gamma \st |\gamma|\leq R} \leq Ce^{CR^\rho}.
\end{equation}
For sufficiently large $|\lambda|$ and $|\mu|$, set $r_\lambda=e^{-\kappa_0|\lambda|^\rho}$ and $t_\mu=e^{-\kappa_1|\mu|^\rho}$, respectively. We first estimate the tails in \eqref{eq:simul-Lambda-Lambda}--\eqref{eq:simul-M-Lambda}. We begin with the proof that
\begin{equation}\label{eq:tail-gamma-diag}
    \lim_{R\to \infty}\sup_{\lambda'\in\Lambda} \sum_{\substack{\lambda\in\Lambda\\ \lambda\ne\lambda'\\ |\lambda|\geq R}} e^{\Omega_1(\lambda')-\Omega_1(\lambda)} e^{2s\kappa_0|\lambda|^\rho} \exp\hp{-a e^{\delta\kappa_0|\lambda|^\rho} |\lambda'-\lambda|^\delta} = 0.
\end{equation}
The $\rho$-separation of $\Lambda$ gives, for $\lambda'\ne\lambda$,
\begin{equation}\label{eq:adm-kappa-gamma}
    e^{\kappa_0|\lambda|^\rho}|\lambda'-\lambda|\geq A_\Lambda\exp\hp{(\kappa_0-B_\Lambda)|\lambda|^\rho}.
\end{equation}

If $|\lambda'|\leq2|\lambda|$, then $\Omega_1(\lambda')-\Omega_1(\lambda)+2s\kappa_0|\lambda|^\rho \leq C|\lambda|^\eta+C|\lambda|^\rho$. Combining this with \eqref{eq:adm-kappa-gamma} and $\kappa_0>B_\Lambda$, we obtain, for all sufficiently large $|\lambda|$,
\begin{align*}
    \Omega_1(\lambda')-\Omega_1(\lambda) +2s\kappa_0|\lambda|^\rho -a e^{\delta\kappa_0|\lambda|^\rho}|\lambda'-\lambda|^\delta &\leq C|\lambda|^\eta+C|\lambda|^\rho -aA_\Lambda^\delta \exp\hp{\delta(\kappa_0-B_\Lambda)|\lambda|^\rho} \\ &\leq -c_1 e^{c_2|\lambda|^\rho},
\end{align*}
for suitable constants $c_1,c_2>0$.

If instead $|\lambda'|>2|\lambda|$, then $|\lambda'-\lambda|\geq |\lambda'|/2$, whence
\begin{equation*}
    \Omega_1(\lambda')-\Omega_1(\lambda) +2s\kappa_0|\lambda|^\rho -ae^{\delta\kappa_0|\lambda|^\rho}|\lambda'-\lambda|^\delta \leq g(|\l^\prime|),
\end{equation*}
where $g(t)=\tau t^\eta+Lt^\rho-a2^{-\delta}e^{\delta\kappa_0|\lambda|^\rho}t^\delta$. For fixed sufficiently large $|\lambda|$, the function $g$ is decreasing on $[2|\lambda|,\infty)$. Indeed, $g^\prime(t) = \tau\eta t^{\eta-1}+L\rho t^{\rho-1}-a\delta2^{-\delta}e^{\delta\kappa_0|\lambda|^\rho}t^{\delta-1}$, and the last term dominates the first two on $[2|\lambda|,\infty)$ once $|\lambda|$ is large, since $\rho<\eta<\delta$. Thus, the preceding expression is bounded above by its value at $t=2|\lambda|$, and hence again by $-c_1 e^{c_2|\lambda|^\rho}$ for all sufficiently large $|\lambda|$. 

Consequently,
\begin{equation*}
    \sup_{\lambda'\in\Lambda} \sum_{\substack{\lambda\in\Lambda\\ \lambda\ne\lambda'\\ |\lambda|\geq R}} e^{\Omega_1(\lambda')-\Omega_1(\lambda)} e^{2s\kappa_0|\lambda|^\rho} \exp\hp{-a e^{\delta\kappa_0|\lambda|^\rho} |\lambda'-\lambda|^\delta} \leq \sum_{\substack{\lambda\in\Lambda\\ |\lambda|\geq R}} e^{-c_1e^{c_2|\lambda|^\rho}} .
\end{equation*}
The right-hand side tends to $0$ as $R\to\infty$, by the counting estimate \eqref{eq:counting-simul}. This proves \eqref{eq:tail-gamma-diag}.

The same argument, applied to $M$, with $\Omega_0$, $\kappa_1$, and no term involving $L$, gives
\begin{equation}\label{eq:tail-M-diag}
    \lim_{R\to \infty}\sup_{\mu'\in M} \sum_{\substack{\mu\in M\\ \mu\ne\mu'\\ |\mu|\geq R}} e^{\Omega_0(\mu')-\Omega_0(\mu)} \exp\hp{-a\hmm{\frac{\mu'-\mu}{t_\mu}}^\delta} =0.
\end{equation}
We next estimate the mixed tails. Since $0<\eta<1$, we have $|\mu|^\eta-|\lambda|^\eta\leq |\mu-\lambda|^\eta$. Therefore
\begin{equation*}
    \Omega_0(\mu)-\Omega_1(\lambda) +2s\kappa_0|\lambda|^\rho -a\hmm{\frac{\mu-\lambda}{r_\lambda}}^\delta \leq -(L-2s\kappa_0)|\lambda|^\rho +\tau|\mu-\lambda|^\eta -a\hmm{\frac{\mu-\lambda}{r_\lambda}}^\delta.
\end{equation*}
Since $r_\lambda\leq1$ and $\eta<\delta$, there exists $C>0$, independent of $\lambda$, such that $\tau |u|^\eta \leq C+\frac a2\hmm{\frac{u}{r_\lambda}}^\delta$ for $u\in\R^d$. Then
\begin{equation}\label{eq:mixed-gamma-pointwise}
    e^{\Omega_0(\mu)-\Omega_1(\lambda)} r_\lambda^{-2s} \exp\hp{-a\hmm{\frac{\mu-\lambda}{r_\lambda}}^\delta} \lesssim e^{-(L-2s\kappa_0)|\lambda|^\rho} \exp\hp{-\frac a2\hmm{\frac{\mu-\lambda}{r_\lambda}}^\delta}.
\end{equation}
In particular,
\begin{equation*}
    e^{\Omega_0(\mu)-\Omega_1(\lambda)} r_\lambda^{-2s} \exp\hp{-a\hmm{\frac{\mu-\lambda}{r_\lambda}}^\delta} \lesssim e^{-(L-2s\kappa_0)|\lambda|^\rho},
\end{equation*}
and by \eqref{eq:simul-summation}
\begin{equation}\label{eq:tail-gamma-mixed}
    \lim_{R\to \infty}\sup_{\mu\in M} \sum_{\substack{\lambda\in\Lambda\\ \lambda\ne\mu\\ |\lambda|\geq R}} e^{\Omega_0(\mu)-\Omega_1(\lambda)} r_\lambda^{-2s} \exp\hp{-a\hmm{\frac{\mu-\lambda}{r_\lambda}}^\delta} =0.
\end{equation}
Similarly, since $0<\rho<\eta<1$, we have $|\lambda|^\eta-|\mu|^\eta\leq |\lambda-\mu|^\eta$ and $|\lambda|^\rho-|\mu|^\rho\leq |\lambda-\mu|^\rho$. Thus
\begin{equation*}
    \Omega_1(\lambda)-\Omega_0(\mu) -2s\kappa_1|\mu|^\rho -a\hmm{\frac{\lambda-\mu}{t_\mu}}^\delta \leq -(2s\kappa_1-L)|\mu|^\rho +C|\lambda-\mu|^\eta+C|\lambda-\mu|^\rho -a\hmm{\frac{\lambda-\mu}{t_\mu}}^\delta .
\end{equation*}
Since $t_\mu\leq1$ and $\rho<\eta<\delta$, we may absorb the two positive powers into half of the last term $C|u|^\eta+C|u|^\rho \leq C_a+\frac a2\hmm{\frac{u}{t_\mu}}^\delta$ for $u\in\R^d$. Therefore
\begin{equation*}
    e^{\Omega_1(\lambda)-\Omega_0(\mu)} t_\mu^{2s} \exp\hp{-a\hmm{\frac{\lambda-\mu}{t_\mu}}^\delta} \lesssim e^{-(2s\kappa_1-L)|\mu|^\rho},
\end{equation*}
and by \eqref{eq:simul-summation}
\begin{equation}\label{eq:tail-M-mixed}
    \lim_{R \to \infty}\sup_{\lambda\in\Lambda} \sum_{\substack{\mu\in M\\ \mu\ne\lambda\\ |\mu|\geq R}} e^{\Omega_1(\lambda)-\Omega_0(\mu)} t_\mu^{2s} \exp\hp{-a\hmm{\frac{\lambda-\mu}{t_\mu}}^\delta} =0.
\end{equation}

Choose $R_\Lambda$ and $R_M$ so large that the tails in \eqref{eq:tail-gamma-diag}, \eqref{eq:tail-gamma-mixed} and \eqref{eq:tail-M-diag}, \eqref{eq:tail-M-mixed} are uniformly smaller than $\frac\varepsilon2$, respectively. Set
\begin{equation*}
    F_\Lambda=\{\lambda\in\Lambda \st |\lambda|<R_\Lambda\}, \qquad F_M=\{\mu\in M \st |\mu|<R_M\}.
\end{equation*}
We will now choose the finitely many exceptional radii. Fix $\lambda\in F_\Lambda$, put $V_\l = (\Lambda \cup M)\setminus\{\lambda\}$ and
\begin{equation*}
    d_\lambda=\min\hch{1,\dist(\lambda,\Lambda\setminus\{\lambda\}),\dist(\lambda,M\setminus\{\lambda\})}>0,
\end{equation*}
where the distance to the empty set is interpreted as $+\infty$. If $v\in V_\lambda$, and $u=|v-\lambda|$, then $u\geq d_\lambda$ and $\Omega_j(v)\leq C_\lambda+Cu^\eta+Cu^\rho$ for $j=0,1$. Consequently, for $0<r\leq1$,
\begin{equation*}
    e^{\Omega_j(v)-\Omega_1(\lambda)} r^{-2s} \exp\hp{-a\hmm{\frac{v-\lambda}{r}}^\delta} \lesssim_\lambda r^{-2s} \exp\hp{h(u)},
\end{equation*}
where $h(u) = Cu^\eta+Cu^\rho-ar^{-\delta}u^\delta$. Since $\rho<\eta<\delta$, the function $h$ is decreasing on $[d_\lambda,\infty)$, provided $r>0$ is sufficiently small. Hence
\begin{equation*}
    \sup_{v\in V_\lambda} e^{\Omega_j(v)-\Omega_1(\lambda)} r^{-2s} \exp\hp{-a\hmm{\frac{v-\lambda}{r}}^\delta} \lesssim_\lambda r^{-2s} \exp\hp{Cd_\lambda^\eta+Cd_\lambda^\rho-ar^{-\delta}d_\lambda^\delta},
\end{equation*}
and the right-hand side goes to $0$ as $r\to0$. Since $F_\Lambda$ is finite, we may choose all $r_\lambda$, $\lambda\in F_\Lambda$, sufficiently small so that
\begin{equation*}
    \sum_{\lambda\in F_\Lambda} \sup_{v\in V_\lambda} e^{\Omega_j(v)-\Omega_1(\lambda)} r_\lambda^{-2s} \exp\hp{-a\hmm{\frac{v-\lambda}{r_\lambda}}^\delta} <\frac\varepsilon2 \quad \text{for }j=0,1.
\end{equation*}
This gives the required finite-term bounds for the estimates from $\Lambda$ to $\Lambda$ and from $\Lambda$ to $M$.

The choice of the exceptional radii $t_\mu$ for $\mu \in F_{M}$ is analogous, because $F_M$ is finite and the additional factor $t^{2s}$ in the mixed term from $M$ to $\Lambda$ is harmless. So we may choose $0<t_\mu\leq1$, $\mu\in F_M$, so small that the finite $M$-term contribution is less than $\frac\varepsilon2$ in each of \eqref{eq:simul-M-M} and \eqref{eq:simul-M-Lambda}.

Combining the tail and finite-term estimates gives all the inequalities \eqref{eq:simul-Lambda-Lambda}--\eqref{eq:simul-M-Lambda}.
\end{proof}

\begin{proof}[Proof of Theorem~\ref{thm:interpol-NUP}]
Fix $\tau>0$. Let $\delta$ be such that $\eta<\delta<1$, and let $\varphi,\psi$ be as in Proposition~\ref{prop:phi-psi}, corresponding to the exponent $\delta$. Thus
\begin{equation*}
    \varphi(0)=1, \quad(-\Delta)^s\varphi(0)=0, \quad\psi(0)=0, \quad(-\Delta)^s\psi(0)=1,
\end{equation*}
there exists an open set $U\Subset\{\xi\in\R^d:1<|\xi|<2\}$ such that $\supp\widehat\varphi\cup\supp\widehat\psi\Subset U$, and for every multiindex $\nu\in\Z_+^d$, there exist constants $C_\nu,c_\nu>0$ such that, for every $x\in\R^d$,
\begin{equation}\label{eq:phi-psi-two-decay}
    |\partial^\nu \varphi(x)|+|\partial^\nu(-\Delta)^s\varphi(x)| + |\partial^\nu \psi(x)|+|\partial^\nu(-\Delta)^s\psi(x)| \le C_\nu e^{-c_\nu|x|^\delta}.
\end{equation}
Fix the constants $C_0,c_0>0$ for the case $|\nu|=0$.

Let $A_\Lambda,B_\Lambda>0$ and $A_M,B_M>0$ be the $\rho$-separation constants for $\Lambda$ and $M$, respectively. By Proposition~\ref{prop:dens-cntg}, there exists $C>0$ such that for $R\geq 1$
\begin{equation*}
    \#(\Lambda\cap B(0,R))+\#(M\cap B(0,R))\lesssim e^{CR^\rho}.
\end{equation*}
This estimate implies that, if $O_{\Lambda,j}=\{\lambda\in\Lambda \st 2^j\le |\lambda|<2^{j+1}\}$, then $\#O_{\Lambda,j}\lesssim e^{C2^{j\rho}}$ and for $a>C$
\begin{equation*}
    \sum_{\lambda\in O_{\Lambda,j}}e^{-a|\lambda|^\rho} \lesssim e^{-(a-C)2^{j\rho}},
\end{equation*}
which is summable in $j$. The same argument applies to $M$. In particular, if $a>C$, then
\begin{equation}\label{eq:rho-summability}
    \sum_{\lambda\in\Lambda}e^{-a|\lambda|^\rho}<\infty, \qquad \sum_{\mu\in M}e^{-a|\mu|^\rho}<\infty.
\end{equation}

Choose $\kappa_0>B_\Lambda$. Let $L>0$ such that $L-2s\kappa_0>C$. Then choose $\kappa_1>B_M$ so large that $2s\kappa_1-L>C$. Now apply Lemma~\ref{lem:existence_rl} with the present sets $\Lambda, M$, with
\begin{equation}\label{eq:choose-par-offdiag}
    \Omega_1(x)=\tau|x|^\eta+L|x|^\rho,\quad\Omega_0(x)=\tau|x|^\eta,\quad a=c_0,\quad\varepsilon=\varepsilon_0,
\end{equation}
where $\varepsilon_0>0$ will be chosen later. By \eqref{eq:rho-summability} and the choices of $L$, $\kappa_0$, $\kappa_1$, the summability assumptions of Lemma~\ref{lem:existence_rl} are fulfilled. Hence there exist finite sets $F_\Lambda\subset\Lambda$, $F_M\subset M$, and radii $0<r_\l,t_\mu\le 1$ for which \eqref{eq:radius-small} holds and the four estimates \eqref{eq:simul-Lambda-Lambda}--\eqref{eq:simul-M-Lambda} are satisfied with $a=c_0$ and $\varepsilon=\varepsilon_0$.

With the scales $r_\lambda$ and $t_\mu$ fixed, define, for $\l \in \Lambda$ and $\mu \in M$,
\begin{equation*}
    \varphi_\l(x)\defi\varphi\hp{\frac{x-\l}{r_\l}}, \qquad  \psi_\mu(x)\defi t_\mu^{2s}\psi\hp{\frac{x-\mu}{t_\mu}}.
\end{equation*}
By Remark~\ref{rem:tranla-dila},
\begin{equation*}
    (-\Delta)^s\varphi_\l(x)=r_\l^{-2s}\hp{(-\Delta)^s\varphi}\hp{\frac{x-\l}{r_\l}},\qquad(-\Delta)^s\psi_\mu(x)=\hp{(-\Delta)^s\psi}\hp{\frac{x-\mu}{t_\mu}},
\end{equation*}
and in particular
\begin{equation}\label{eq:normal-two-set}
    \varphi_\l(\l)=1,\quad(-\Delta)^s\varphi_\l(\l)=0,\quad\psi_\mu(\mu)=0,\quad(-\Delta)^s\psi_\mu(\mu)=1.
\end{equation}
Also, by the chain rule and \eqref{eq:phi-psi-two-decay}, for every multiindex $\nu\in\Z_+^d$,
\begin{align}
    |\partial^\nu \varphi_\lambda(x)|
    &\lesssim_\nu r_\lambda^{-|\nu|} \exp\!\hp{-c_\nu\hmm{\frac{x-\lambda}{r_\lambda}}^\delta},
    & |\partial^\nu \psi_\mu(x)| &\lesssim_\nu t_\mu^{2s-|\nu|} \exp\!\hp{-c_\nu \hmm{\frac{x-\mu}{t_\mu}}^\delta},\label{eq:scaled-der}
    \\
    |\partial^\nu (-\Delta)^s\varphi_\lambda(x)|
    &\lesssim_\nu r_\lambda^{-|\nu|-2s}\exp\!\hp{-c_\nu\hmm{\frac{x-\lambda}{r_\lambda}}^\delta},
    & |\partial^\nu (-\Delta)^s\psi_\mu(x)| 
    &\lesssim_\nu t_\mu^{-|\nu|} \exp\!\hp{-c_\nu\hmm{\frac{x-\mu}{t_\mu}}^\delta}.\label{eq:scaled-frac-der}
\end{align}
For $(\alpha,\beta)\in X_{\tau,L}(\Lambda, M)$, define the synthesis operator
\begin{equation*}
    P(\alpha,\beta)(x) = \sum_{\l\in\Lambda}\alpha(\l)\varphi_\l(x) + \sum_{\mu\in M}\beta(\mu)\psi_\mu(x).
\end{equation*}
We will prove that the series above converges in $\Sp(\R^d)$ and satisfies the decay estimate
\begin{equation}\label{eq:P-op-decay}
    |\partial^\nu P(\alpha,\beta)(x)| \le C_\nu\hn{(\alpha,\beta)}_{X_{\tau,L}}e^{-c_\nu|x|^\eta} \qquad \text{for } x\in\R^d.
\end{equation}
Indeed, by definition of the norm in $X_{\tau,L}(\Lambda, M)$,
\begin{equation}\label{eq:a-b-pointw}
    |\alpha(\l)|\le \hn{(\alpha,\beta)}_{X_{\tau,L}}e^{-\Omega_1(\l)}, \qquad |\beta(\mu)|\le \hn{(\alpha,\beta)}_{X_{\tau,L}}e^{-\Omega_0(\mu)}.
\end{equation}
For $N\in\Z_+$ and $\nu\in\Z_+^d$, write $p_{N,\nu}(f)=\sup_{x\in\R^d}(1+|x|)^N|\partial^\nu f(x)|$. Let $\lambda\in\Lambda$ and set $y_\lambda=\frac{x-\lambda}{r_\lambda}$. Since $0<r_\lambda\le1$, we have 
\begin{equation*}
    (1+|x|) \lesssim (1+|\lambda|)(1+|y_\lambda|).
\end{equation*}
Using this, \eqref{eq:scaled-der}, and \eqref{eq:a-b-pointw},
\begin{align*}
    p_{N,\nu}\hp{\alpha(\lambda)\varphi_\lambda} &\lesssim_{N,\nu}|\alpha(\lambda)|r_\lambda^{-|\nu|} (1+|\lambda|)^N \sup_{y\in\R^d}(1+|y|)^N e^{-c_\nu |y|^\delta} \nonumber\\ &\lesssim_{N,\nu} \hn{(\alpha,\beta)}_{X_{\tau,L}} e^{-\Omega_1(\lambda)} r_\lambda^{-|\nu|} (1+|\lambda|)^N .
\end{align*}
Since $r_\lambda\le1$, we may dominate $r_\lambda^{-|\nu|}$ by $r_\lambda^{-|\nu|-2s}$. By Lemma~\ref{lem:properties_rl}~\hyperref[item:absor-scale]{$(i)$}, applied with $K=|\nu|+2s$, and by the elementary estimate $(1+|\lambda|)^N\le C_\varepsilon e^{\varepsilon|\lambda|^\eta}$, we get, choosing $\varepsilon>0$ sufficiently small and using $e^{-\Omega_1(\lambda)} \le e^{-\tau|\lambda|^\eta}$,
\begin{equation*}
p_{N,\nu}\hp{\alpha(\lambda)\varphi_\lambda}\lesssim_{N,\nu}\hn{(\alpha,\beta)}_{X_{\tau,L}}e^{-c|\lambda|^\eta}.
\end{equation*}
The estimate for the $M$-terms is similar, and we get
\begin{equation*}
p_{N,\nu}\hp{\beta(\mu)\psi_\mu}\lesssim_{N,\nu}\hn{(\alpha,\beta)}_{X_{\tau,L}}e^{-c|\mu|^\eta}.
\end{equation*}
Since $\eta>\rho$, Proposition~\ref{prop:cntg-sum} and the previous two inequalities imply absolute convergence of $P(\alpha,\beta)$ in every Schwartz seminorm. Hence the series $P(\alpha,\beta)$ converges in $\Sp(\R^d)$, and the map
\begin{equation*}
    P:X_{\tau,L}(\Lambda,M)\to \Sp(\R^d)
\end{equation*}
is continuous. It remains to prove the pointwise decay. Termwise differentiation is justified by the absolute convergence proved above. By \eqref{eq:scaled-der}, \eqref{eq:a-b-pointw}, $r_\lambda^{-|\nu|}\le r_\lambda^{-|\nu|-2s}$, and $t_\mu^{2s-|\nu|}\le t_\mu^{-|\nu|-2s}$, we get
\begin{align*}
    |\partial^\nu P(\alpha,\beta)(x)| &\lesssim_\nu \hn{(\alpha,\beta)}_{X_{\tau,L}} \sum_{\lambda\in\Lambda} e^{-\Omega_1(\lambda)} r_\lambda^{-|\nu|-2s} \exp\hp{-c_\nu\hmm{\frac{x-\lambda}{r_\lambda}}^\delta} \\
    &\quad+ \hn{(\alpha,\beta)}_{X_{\tau,L}} \sum_{\mu\in M} e^{-\Omega_0(\mu)} t_\mu^{-|\nu|-2s} \exp\hp{-c_\nu\hmm{\frac{x-\mu}{t_\mu}}^\delta}.
\end{align*}
Since $e^{-\Omega_1(\lambda)}\le e^{-\tau|\lambda|^\eta}$ and $e^{-\Omega_0(\mu)}=e^{-\tau|\mu|^\eta}$, Lemma~\ref{lem:properties_rl}~\hyperref[item:decay-scale]{$(ii)$}, applied separately to $\Lambda$ and to $M$ with $K=|\nu|+2s$, $\sigma=\tau$, and $a=c_\nu$, gives
\begin{equation*}
    |\partial^\nu P(\alpha,\beta)(x)| \lesssim_\nu\hn{(\alpha,\beta)}_{X_{\tau,L}}e^{-c_\nu|x|^\eta}
\end{equation*}
after possibly decreasing $c_\nu>0$. This proves \eqref{eq:P-op-decay}. Define also
\begin{equation}\label{eq:defi-H}
    H(\alpha,\beta) = \sum_{\l\in\Lambda}\alpha(\l)(-\Delta)^s\varphi_\l + \sum_{\mu\in M}\beta(\mu)(-\Delta)^s\psi_\mu.
\end{equation}
Using the same argument as for $P(\alpha, \beta)$, but with \eqref{eq:scaled-frac-der}, \eqref{eq:a-b-pointw}, and again Lemma~\ref{lem:properties_rl}~\hyperref[item:decay-scale]{$(ii)$}, we obtain that the series in \eqref{eq:defi-H} converges in $\Sp(\R^d)$ and
\begin{equation}\label{eq:H-decay}
    |\partial^\nu H(\alpha,\beta)(x)| \lesssim_\nu\hn{(\alpha,\beta)}_{X_{\tau,L}}e^{-c_\nu|x|^\eta}.
\end{equation}
Since $\supp\widehat\varphi\cup\supp\widehat\psi\Subset\{\xi\in\R^d:1<|\xi|<2\}$ and $0<r_\l,t_\mu\le 1$, translations and dilations give
\begin{equation*}
    \supp\widehat{\varphi_\l}\cup\supp\widehat{\psi_\mu}\subset \{\xi\in\R^d \st |\xi|\ge 1\}.
\end{equation*}
Choose $\chi\in C^\infty(\R^d)$ such that $\chi=0$ in a neighbourhood of the origin and $\chi=1$ on $\{|\xi|\ge 1\}$, and set $m(\xi)=\chi(\xi)|\xi|^{2s}$. Since $m\in C^{\infty}(\R^d)$ and all derivatives have polynomial growth, $T_mf=\mathcal F^{-1}(m\widehat f)$ is continuous on $\Sp(\R^d)$. For every finite partial sum $P_F$ of $P$, all Fourier supports stay in $\{|\xi|\ge 1\}$, so $T_mP_F=(-\Delta)^sP_F$. On the other hand, $T_mP_F$ is exactly the corresponding finite partial sum of \eqref{eq:defi-H}. Passing to the limit in $\Sp(\R^d)$ yields
\begin{equation}\label{eq:H-is-frac}
(-\Delta)^sP(\alpha,\beta) \in\Sp(\R^d),\qquad (-\Delta)^sP(\alpha,\beta)=H(\alpha,\beta).
\end{equation}
Define the trace operator by $ \mathcal{T}f = \hp{f|_\Lambda,(-\Delta)^sf|_M}$. By \eqref{eq:normal-two-set},
\begin{equation*}
     \mathcal{T} P=I+E
\end{equation*}
on $X_{\tau,L}(\Lambda,M)$, where $E= (E_0,E_1)$ is given by
\begin{equation}\label{eq:E-zero}
    E_0(\alpha,\beta)(\l') = \sum_{\substack{\l\in\Lambda\\ \l\ne\l'}}\alpha(\l)\varphi_\l(\l') + \sum_{\substack{\mu\in M\\ \mu\ne\l'}}\beta(\mu)\psi_\mu(\l') \quad \text{for } \l'\in\Lambda,
\end{equation}
and
\begin{equation*}
    E_1(\alpha,\beta)(\mu') = \sum_{\substack{\l\in\Lambda\\ \l\ne\mu'}}\alpha(\l)(-\Delta)^s\varphi_\l(\mu') + \sum_{\substack{\mu\in M\\ \mu\ne\mu'}}\beta(\mu)(-\Delta)^s\psi_\mu(\mu') \quad \text{for } \mu'\in M.
\end{equation*}
The exclusions in the mixed sums only remove diagonal terms which vanish by \eqref{eq:normal-two-set} when the two sets intersect. Using \eqref{eq:scaled-der} with $\nu=0$, $1\leq r_\lambda^{-2s}$, \eqref{eq:a-b-pointw}, and \eqref{eq:E-zero}, we obtain
\begin{align*}
    e^{\Omega_1(\l')}|E_0(\alpha,\beta)(\l')| &\le C_0\hn{(\alpha,\beta)}_{X_{\tau,L}} \sum_{\substack{\l\in\Lambda\\ \l\ne\l'}} e^{\Omega_1(\l')-\Omega_1(\l)} r_\l^{-2s} \exp\hp{-c_0\hmm{\frac{\l'-\l}{r_\l}}^\delta}\\
    &\quad+ C_0\hn{(\alpha,\beta)}_{X_{\tau,L}} \sum_{\substack{\mu\in M\\ \mu\ne\l'}} e^{\Omega_1(\l')-\Omega_0(\mu)} t_\mu^{2s} \exp\hp{-c_0\hmm{\frac{\l'-\mu}{t_\mu}}^\delta}.
\end{align*}
Taking the supremum over $\lambda'$ and using the choice of parameters in \eqref{eq:choose-par-offdiag}, the estimates \eqref{eq:simul-Lambda-Lambda} and \eqref{eq:simul-M-Lambda} from Lemma~\ref{lem:existence_rl} yield
\begin{equation*}
    \sup_{\l'\in\Lambda}e^{\Omega_1(\l')}|E_0(\alpha,\beta)(\l')|\le2C_0\varepsilon_0\hn{(\alpha,\beta)}_{X_{\tau,L}}.
\end{equation*}
Similarly, using \eqref{eq:scaled-frac-der} with $\nu=0$, \eqref{eq:simul-Lambda-M}, and \eqref{eq:simul-M-M},
\begin{equation*}
    \sup_{\mu'\in M}e^{\Omega_0(\mu')}|E_1(\alpha,\beta)(\mu')| \le 2C_0\varepsilon_0\hn{(\alpha,\beta)}_{X_{\tau,L}}.
\end{equation*}
Choose $\varepsilon_0>0$ so small that $\varepsilon_0\le \frac1{8C_0}$. Then $\hn{E}_{X_{\tau,L}(\Lambda, M)\to X_{\tau,L}(\Lambda, M)}\le \frac12$. Hence $I+E$ is invertible on $X_{\tau,L}(\Lambda, M)$. Finally, define
\begin{equation*}
    Q_{\tau,L}= P(I+E)^{-1}.
\end{equation*}
Then $Q_{\tau,L}:X_{\tau,L}(\Lambda, M)\to\Sp(\R^d)$ is continuous and linear. Since $(-\Delta)^sP$ takes values in $\Sp(\R^d)$ by \eqref{eq:H-is-frac}, the same is true for $Q_{\tau,L}$. Moreover,
\begin{equation*}
    \mathcal{T} Q_{\tau,L} =  \mathcal{T} P(I+E)^{-1} = (I+E)(I+E)^{-1} = I,
\end{equation*}
which proves \eqref{eq:two-set-interpolation}. It remains only to show the decay estimate for $Q_{\tau, L}$. Let $(\widetilde\alpha,\widetilde\beta)=(I+E)^{-1}(\alpha,\beta)$. Then $Q_{\tau,L}(\alpha,\beta)=P(\widetilde\alpha,\widetilde\beta)$ and $(-\Delta)^sQ_{\tau,L}(\alpha,\beta)=H(\widetilde\alpha,\widetilde\beta)$. Therefore \eqref{eq:P-op-decay}, \eqref{eq:H-decay}, and the boundedness of $(I+E)^{-1}$ on $X_{\tau,L}(\Lambda, M)$ give
\begin{equation*}
    |\partial^\nu Q_{\tau,L}(\alpha,\beta)(x)| + |\partial^\nu(-\Delta)^sQ_{\tau,L}(\alpha,\beta)(x)| \lesssim_\nu\hn{(\alpha,\beta)}_{X_{\tau,L}}e^{-c_\nu|x|^\eta}
\end{equation*}
for $x\in \R^d$. This proves \eqref{eq:interp-decay}.
\end{proof}

\subsection{Proof of Theorem~\ref{thm:nup-strong}}\label{sec:4.4} 

We start by recalling the Gelfand–Shilov-type spaces defined in \eqref{eq:GS-spaces}. So far, it is not clear whether they contain a nonzero function.
\begin{remark}\label{rem:GS-nonempt}
Let $0<s<1$. The range $0<p,q<1$ is sharp for the nontriviality of $\mathcal{S}_{s}^{p,q}(\R^d)$. Indeed, if either $p\geq 1$ or $q\geq 1$, the proof of Theorem~\ref{thm:UP}~\hyperref[item:1-thmUP]{$(i)$} shows that $\mathcal{S}_{s}^{p,q}(\R^d) = \hch{0}$. Conversely, suppose that $0<p,q<1$ and set $\gamma=\max\{p,q\}$. For $\varphi\neq 0$ as in Proposition~\ref{prop:phi-psi}, we have $\varphi \in \mathcal{S}_{s}^{\gamma,\gamma}(\R^d)\subset \mathcal{S}_{s}^{p,q}(\R^d)$.
\end{remark}

\begin{proof}[Proof of Theorem~\ref{thm:nup-strong}]
Set $r=\max\hch{p,q}$ and choose $\max\hch{\rho,r}<\eta<\delta<1$. By Proposition~\ref{prop:Y-ela-two-sets}, we can choose an infinite set $Y=\hch{y_j\st j\geq1}\subset\R^d\setminus(\Lambda\cup M)$ such that $\Gamma\defi\Lambda\cup Y$ is $\rho$-separated. Fix $\tau>0$ and apply Theorem~\ref{thm:interpol-NUP} to the pair $(\Gamma,M)$. Let $Q_{\tau, L}$ be the corresponding interpolation operator. For each $j\geq1$, define $(\alpha_j,\beta_j)\in X_{\tau,L}(\Gamma,M)$ by
\begin{equation*}
\alpha_j= \mathbf{1}_{\hch{\gamma = y_j}}, \qquad\beta_j\equiv0.
\end{equation*}
Set $f_j=Q_{\tau,L}(\alpha_j,\beta_j)$. The interpolation identities give    $f_j(\lambda)=0$ for $\lambda\in\Lambda$ and $(-\Delta)^sf_j(\mu)=0$ for $\mu\in M$. On the other hand, $f_j(y_i)=\mathbf{1}_{\hch{i=j}},$ $i,j\geq1$. Hence $\hch{f_j\st j\geq1}$ is linearly independent. It remains to check that $f_j\in\mathcal S_s^{p,q}(\R^d)$. Since $\eta>r$, \eqref{eq:interp-decay} gives $f_j\in\mathcal S_s^{p,q}(\R^d)$.
\end{proof}

\section{Analytic non-uniqueness results}\label{sec:5}

In this section, we establish the analytic $\NUP$ results. For Theorem~\ref{thm:BL-nup}, our approach is based on Fourier analysis and exploits the orthogonality of exponential functions in $L^2$-spaces. For Theorem~\ref{thm:debranges-nup}, we use results from the theory of de~Branges spaces, which will be defined in Subsection~\ref{sec:5.3}, to construct the required examples. 

\subsection{Proof of Theorem~\ref{thm:BL-nup}}\label{sec:5.1}
We first establish a few preliminary facts. Consider $f\in \mathcal{S}(\R^d)$, $M\in \GL_d(\R)$ and $Q= [0,1)^d$. Define
\begin{equation}\label{eq:def-Q-M}
    \Lambda= M\Z^d, \quad \Lambda^*= M^{-t}\Z^d, \quad Q_M= M^{-t}Q.
\end{equation}
Because $Q_M$ is a fundamental domain for the dual lattice $\Lambda^*$, the family $\{Q_M+\ell \st \ell\in\Lambda^*\}$ tiles $\R^d$. Therefore, for every $k\in \Z^d$, a change of variables combined with Fubini's theorem yields
\begin{equation*}
    f(Mk) = \sum_{\ell\in\Lambda^*} \int_{Q_M+\ell} e^{2\pi i Mk\cdot w} \widehat{f}(w)\dw = \int_{Q_M} \sum_{\ell\in\Lambda^*} e^{2\pi i Mk\cdot(\xi+\ell)} \widehat{f}(\xi+\ell)\dxi.
\end{equation*}
For $\ell\in \Lambda^*$, we have $M^t\ell\in \Z^d$. Thus $Mk\cdot \ell = k\cdot M^t\ell \in \Z$ for every $k\in \Z^d$, which implies $e^{2\pi i Mk\cdot \ell}=1$. Then
\begin{equation}\label{eq:f-per-QM}
    f(Mk) = \int_{Q_M} e^{2\pi i Mk\cdot \xi} \hp{\sum_{\ell\in\Lambda^*}\widehat{f}(\xi+\ell)} \dxi,
\end{equation}
and in the same way
\begin{equation}\label{eq:FL-f-per-QM}
    (-\Delta)^s f(Mk)=\int_{Q_M} e^{2\pi i Mk\cdot \xi} \hp{\sum_{\ell\in\Lambda^*} |\xi+\ell|^{2s}\widehat f(\xi+\ell)}\dxi.
\end{equation}

Motivated by the previous identities, consider the unitary linear map 
\begin{equation*}
    U:L^2(Q_M)\to L^2(Q), \qquad (Uh)(\xi) = |\det M|^{-1/2}h(M^{-t}\xi). 
\end{equation*}
Then for every $k\in\Z^d$ we have $U(e^{2\pi i Mk\cdot \xi})(w)=|\det M|^{-1/2} e^{2\pi i k\cdot w}$. Since the family $\{e^{2\pi i k\cdot w}\}_{k\in\Z^d}$ is complete in $L^2(Q)$, it follows by unitarity of $U$ that $\{e^{2\pi i Mk\cdot w}\}_{k\in\Z^d}$ is complete in $L^2(Q_M)$.

In particular, the vanishing of the sequences $\{f(Mk)\}_{k\in\Z^d}$ and $\{(-\Delta)^s f(Mk)\}_{k\in\Z^d}$ is equivalent to the vanishing, for a.e. $\xi\in Q_M$, of the two periodized functions
\begin{equation}\label{eq:def-F-G-per}
    F_M(\xi)\defi \sum_{\ell\in\Lambda^*}\widehat{f}(\xi+\ell),\qquad G_M(\xi)\defi {\sum_{\ell\in\Lambda^*} |\xi+\ell|^{2s}\widehat f(\xi+\ell)}.
\end{equation}
Let $v$ be the nonzero vector given by $v= M^{-t}e_1\in \Lambda^*$. For $\xi\in Q_M$, define
\begin{equation*}
    a_1(\xi)=\widehat f(\xi+v), \quad a_2(\xi)=\widehat f(\xi+2v), \quad a_3(\xi)=\widehat f(\xi+3v).
\end{equation*}
If $\operatorname{supp}\widehat f$ is contained in
\begin{equation}\label{eq:def-BM}
    B_M = (Q_M+v)\cup(Q_M+2v)\cup(Q_M+3v),
\end{equation}
then the only possibly nonzero terms in the sums defining $F_M(\xi)$ and $G_M(\xi)$ are precisely those corresponding to $v,2v,3v$. Thus, the conditions $F_M(\xi)=0$ and $G_M(\xi)=0$ become the system
\begin{equation}\label{eq:syst-a_j}
\left\{
\begin{aligned}
a_1(\xi) + a_2(\xi) + a_3(\xi) &= 0, \\
|\xi+v|^{2s} a_1(\xi) + |\xi+2v|^{2s} a_2(\xi) + |\xi+3v|^{2s} a_3(\xi) &= 0.
\end{aligned}
\right.
\end{equation}

Now we will construct nontrivial smooth functions $a_1, a_2, a_3$ on $Q_M$ that satisfy the given linear relations and ensure that the resulting patch $\widehat{f}$ belongs to $\Sp(\R^d)$. The main idea is that, for each fixed $\xi \in Q_M$, the unknowns $a_1, a_2, a_3$ satisfy two linear homogeneous equations, so the solution space has dimension at least one. We exhibit a convenient one-parameter family of solutions.

\begin{proof}[Proof of Theorem~\ref{thm:BL-nup}]
Let $M\in \GL_d(\R)$, $Q= [0,1)^d$, and consider again the sets defined in \eqref{eq:def-Q-M}. 

Let $v= M^{-t}e_1\in \Lambda^*$ and $B_M$ as in \eqref{eq:def-BM}. Note that the set $H_v=\{\xi\in\R^d\st |\xi+2v|=|\xi+v|\}$ is a proper affine hyperplane. Indeed, we have $|\xi+2v|^2-|\xi+v|^2=2\xi\cdot v + 3|v|^2$, so
\begin{equation*}
    H_v=\hch{\xi\in\R^d\st \xi\cdot v=-\tfrac{3}{2}|v|^2},
\end{equation*}
which is an affine hyperplane because $v\neq 0$. In particular, there exists a nonempty open set $O$ with $\overline{O} \subset \operatorname{int}(Q_M)\setminus H_v$. Choose a nonzero smooth function $\varphi\in C_c^\infty(O)$, and define $a_3(\xi)=\varphi(\xi)$. We should have by \eqref{eq:syst-a_j} that $a_1(\xi)=-a_2(\xi)-a_3(\xi)$. Substituting this into \eqref{eq:syst-a_j} gives
\begin{equation*}
    \hp{|\xi+2v|^{2s}-|\xi+v|^{2s}} a_2(\xi) +\hp{|\xi+3v|^{2s}-|\xi+v|^{2s}}a_3(\xi)=0.
\end{equation*}
Therefore, the last equation can be solved for $a_2$
\begin{equation}\label{eq:def-a2}
    a_2(\xi) = -\varphi(\xi)\hp{\frac{|\xi+3v|^{2s}-|\xi+v|^{2s}}{|\xi+2v|^{2s}-|\xi+v|^{2s}}}
\end{equation}
and then
\begin{equation}\label{eq:def-a1}
    a_1(\xi) = -a_2(\xi)-a_3(\xi) = -\varphi(\xi)\hp{1-\frac{|\xi+3v|^{2s}-|\xi+v|^{2s}}{|\xi+2v|^{2s}-|\xi+v|^{2s}}}.
\end{equation}
Thus, the quotients on the right-hand side of \eqref{eq:def-a2} and \eqref{eq:def-a1} are smooth functions of $\xi$ on $O$ (on a neighborhood of $\supp \varphi \subset O$, both numerator and denominator are smooth, and the denominator never vanishes). After extending $a_1,a_2,a_3$ by zero outside $O$, we define $\widehat f$ on $\R^d$ by
\begin{equation*}
    \widehat f(\xi)=\sum_{j=1}^3 a_j(\xi-jv).
\end{equation*}
(observe that $kv + O\subset B_M$, $k=1,2,3$). Since $a_1,a_2,a_3$ are in $C_c^\infty(O)$, patching these three pieces together yields a $C^\infty$ function compactly supported inside $\operatorname{int}(B_M)$. Then $\widehat f\in \mathcal S(\R^d)$ and $f\not\equiv0$. By construction, for every $\xi\in Q_M$ the triple $(a_1(\xi),a_2(\xi),a_3(\xi))$ satisfies \eqref{eq:syst-a_j} pointwise. Therefore the periodized functions $F_M(\xi)$ and $G_M(\xi)$ defined in \eqref{eq:def-F-G-per} vanish pointwise on $Q_M$. Plugging back into \eqref{eq:f-per-QM} and \eqref{eq:FL-f-per-QM}, we obtain Theorem~\ref{thm:BL-nup}.
\end{proof}

\subsection{Density and uniformly separated conditions}\label{sec:5.2}
We first introduce in this subsection the upper Beurling density with respect to a measure.
\begin{definition}\label{def:dens-meas}
Let $\mu$ be a positive Borel measure on $\R$, and let $\Gamma\subset\R$ be discrete. For $r>0$
\begin{equation*}
    m(\Gamma,r) \defi \hch{\frac{\#(\Gamma\cap I)}{\mu(I)}\st I\subset\R \text{ an interval and } \mu(I)=r},
\end{equation*}
and we define the upper and lower Beurling densities of $\Gamma$ with respect to $\mu$, respectively, by
\begin{equation}\label{eq:def-meas-D}
    \D_\mu^+(\Gamma) = \limsup_{r\to\infty} \sup m(\Gamma,r), \qquad \D_\mu^-(\Gamma) = \liminf_{r\to\infty} \inf m(\Gamma,r).
\end{equation}
\end{definition}
We next relate \eqref{eq:dens-with-F} and Definition~\ref{def:dens-meas}, which agree in the following special case. If $\varphi\in C^1(\R)$ is an increasing\footnote{Throughout this work, increasing means strictly increasing.} function and $\dmu_\varphi(x)=\varphi'(x)\dx$, then for every interval $I=[a,b]\subset\R$ we have 
\begin{equation*}
    \mu_\varphi(I)=\int_a^b \dmu_\varphi(x)=\varphi(b)-\varphi(a)=|\varphi(I)|.
\end{equation*}
Moreover, strict monotonicity implies that $\varphi$ is injective and forms a bijection with its image, so that $ \#(\Gamma\cap I)=\#(\varphi(\Gamma)\cap \varphi(I))$. Consequently, by \eqref{eq:def-meas-D} we have
\begin{equation}\label{eq:D-are-equiv}
    \D_{\mu_\varphi}^+(\Gamma)= \limsup_{r\to\infty} \sup\hch{\frac{\#(\varphi(\Gamma)\cap J)}{|J|}\st J\subset \varphi(\R)\ \text{interval},\ |J|=r} = \D_\varphi^+(\Gamma),
\end{equation}
where $\D_{\varphi}^+$ was defined in \eqref{eq:dens-with-F}. The same holds for $\D_{\mu_\varphi}^-(\Gamma)$. In particular, if $\varphi(\R)=\R$, then it follows from \eqref{eq:D-are-equiv} that
\begin{equation}\label{eq:D-meas-func}
    \D_{\mu_\varphi}^\pm(\Gamma)= \D_\varphi^\pm(\Gamma)=\D^\pm(\varphi(\Gamma)),
\end{equation}
where $\D^\pm$ denote the classical Beurling densities, corresponding to the choice $F(x)=x$ in \eqref{eq:dens-with-F}.

As we shall see in the next subsection, Definition~\ref{def:dens-meas} has a natural meaning when the measure $\mu_\varphi$ is doubling. We therefore begin by recalling the notion of a doubling measure, and then introduce the uniform separation conditions used throughout this work.
\begin{definition}[Doubling measure]
A measure $\mu$ on the real line is called doubling if there exists a universal constant $C>0$ such that for every bounded interval $I \subset \mathbb{R}$ we have $\mu(2I) \leq C \mu(I)$, where $2I$ denotes the interval with the same center as $I$ and twice its radius.
\end{definition}
\begin{example}[{\cite[Ex.~7.1.6]{grafakosclassical}}]\label{ex:2.4}
For $\gamma>-1$ the measure $\dnu(x) = (1+|x|)^\gamma\dx$ is doubling on $\mathbb{R}$.
\end{example}
\begin{remark}\label{rem:2.5}
A well-known fact about doubling measures is that the doubling is preserved under pointwise comparability. More precisely, let $f, w : \mathbb{R} \to [0,\infty)$ be continuous, and define measures $\mu$ and $\nu$ by $\dmu = f\dx$ and $\dnu = w\dx$. If $f \asymp w$ and $\nu$ is doubling, then $\mu$ is doubling.
\end{remark}
We also define the notion of uniformly discrete sequences.
\begin{definition}\label{def:2.6}
Let $\Gamma \subset \R$ be a sequence. We call $\Gamma$ uniformly discrete if there exists a constant $\delta>0$ such that $|\gamma - \gamma^\prime| \ge \delta$ for all $\gamma,\gamma^\prime \in \Gamma$ with $\gamma\neq \gamma^\prime$.
\end{definition}

\begin{remark}\label{rem:2.7}
Let $\Gamma\subset\R$ be a sequence with $\D^{+}(\Gamma)<\infty$. Then $\Gamma$ is a finite union of uniformly discrete sequences. Indeed, $\D^{+}(\Gamma)<\infty$ implies the uniform local bound $\sup_{x\in\mathbb{R}}\#(\Gamma\cap [x,x+1])<\infty$, and from this one obtains a decomposition of $\Gamma$. More precisely, there exists $N\in\N$ and sequences 
$\Gamma_1, \dots, \Gamma_N \subset \Gamma$ such that
\begin{equation*}
  \Gamma = \Gamma_1 \cup \dots \cup \Gamma_N, \qquad \Gamma_i \cap \Gamma_j = \emptyset \quad \text{for } i \neq j,
\end{equation*}
and there exists $\delta_i>0$ such that the condition in Definition~\ref{def:2.6} holds for $\Gamma_i$.
\end{remark}

\subsection{de~Branges and model spaces}\label{sec:5.3}
We begin by recalling the basic structure of de~Branges spaces. A function $E$ is called Hermite--Biehler if $E$ is entire, which means $E \in \mathscr{H}(\C)$, and 
\begin{equation*}
    |E^{\ast}(z)| < |E(z)| \quad \text{for } z\in \C_+,
\end{equation*}
where $E^{\ast}(z) \defi \overline{E(\bar z)}$. If $E$ is Hermite--Biehler, the corresponding de~Branges space is defined by
\begin{equation*}
    \H(E)\defi \hch{F \in \mathscr{H}(\C) \st F/E\in H^2(\C_+) \text{ and } F^{\ast}/E\in H^2(\C_+)},
\end{equation*}
equipped with the norm
\begin{equation*}
    \hn{F}_{\H(E)}^2 = \int_{\R}|F(x)|^2|E(x)|^{-2}\dx.
\end{equation*}
It is a standard theorem of de~Branges that $\mathcal{H}(E)$ is a reproducing kernel Hilbert space \cite[Theorem~19]{deB}. These spaces can be viewed as a generalization of the Paley--Wiener spaces, which arise as de~Branges spaces corresponding to $E(z)=e^{-iaz}$, $a>0$.

The next result, due to Marzo, Nitzan, and Olsen \cite{Marzo2012}, provides a sufficient condition for a real sequence to be interpolating in a de~Branges space under a doubling assumption on the phase measure. Before stating it, we recall the relevant notions. Let $E$ be a Hermite--Biehler function with phase function $\varphi$, meaning that $\varphi$ is the smooth real increasing function such that 
\begin{equation*}
    E(x)=|E(x)|\exp\hp{-i\varphi(x)} \quad \text{ for } x \in \R.
\end{equation*}

A real sequence $\Gamma \subset \mathbb{R}$ is said to be \emph{uniformly $\varphi$-discrete} if $\varphi(\Gamma)$ is uniformly discrete in the sense of Definition~\ref{def:2.6}, and its upper $\varphi$-density is given by \eqref{eq:D-are-equiv}.

\begin{theoremA}[{\cite[Theorem 1.5]{Marzo2012}}]\label{thm:A}
Let $E$ be a Hermite--Biehler function and $\H(E)$ the associated de~Branges space with phase function $\varphi$. For $\omega \in \C$, denote the reproducing kernel of $\H(E)$ at $\omega$ by $k_\omega(\cdot)$. Suppose that $\varphi'(x)\dx$ is a doubling measure on $\mathbb{R}$ and let $\Gamma\subset\R$ be a real sequence. 

If $\Gamma$ is uniformly $\varphi$-discrete and $\D^+_{\varphi}(\Gamma) < \pi^{-1}$, then $\Gamma$ is interpolating for $\H(E)$. More precisely, for every sequence $\{a_\gamma\}_{\gamma\in\Gamma}$ satisfying
\begin{equation*}
    \sum_{\gamma \in \Gamma} \frac{|a_\gamma|^2 }{\hn{k_\gamma}^2} < \infty,
\end{equation*}
there exists a function $g \in \H(E)$ such that $g(\gamma) = a_\gamma$ for every $\gamma\in\Gamma$.
\end{theoremA}

We now introduce the meromorphic inner function, and hence the Hermite--Biehler function, to which Theorem~\hyperref[thm:A]{A} will be applied. Let $\{z_n\}_{n\ge 1}\subset \C_+$ be a sequence accumulating only at infinity and satisfying the Blaschke condition
\begin{equation}\label{eq:def-Blasc-cond}
    \sum_{n\ge 1}\frac{\IM (z_n)}{|z_n|^2}<\infty.
\end{equation}
For a fixed $\ell\in\N$, we define
\begin{equation}\label{eq:def-Theta-ell}
    \Theta_\ell(z)\defi \hbra{\prod_{n=1}^\infty \frac{1-z/z_n}{1-z/\overline{z_n}}}^\ell.
\end{equation}
Thus, $\Theta_\ell$ is the Blaschke product with zeros at the points $z_n$, each of multiplicity $\ell$. In particular, it follows from \cite[Problem~23]{deB} that $\Theta_{\ell}$ is an inner function, that is, $\Theta_{\ell}$ is bounded in the upper half-plane and $|\Theta_{\ell}|=1$ on $\R$. Hence, by the standard de~Branges correspondence for meromorphic inner functions \cite{Havin_Mashreghi_2003}, there exists a Hermite--Biehler entire function $E_{\ell}$, with no real zeros, such that
\begin{equation}\label{eq:Theta-ell-E}
    \Theta_{\ell}(z) = \frac{E_{\ell}^{\ast}(z)}{E_{\ell}(z)}.
\end{equation}
Let $\varphi_{\ell}$ denote the increasing phase function of $E_{\ell}$, so that $\Theta_{\ell}(x) = e^{i 2 \varphi_{\ell}(x)}$. Since the zeros in the factorization of $\Theta_{\ell}$ accumulate only at infinity, we may invoke the Poisson kernel representation for the arguments of meromorphic inner functions. Thus, we obtain\footnote{Equivalently, one may differentiate logarithmically the factorization of $\Theta_{\ell}$, using that the corresponding Blaschke product converges locally uniformly in a neighborhood of $\R$.}
\begin{equation}\label{eq:Der-Theta-ell}
    |\Theta_{\ell}'(x)| = 2\varphi_{\ell}'(x) = 2\ell S(x), \quad \text{where } S(x) \defi \sum_{n\geq1} \frac{\IM (z_n)}{|x-z_n|^2}.
\end{equation}

We now turn to the model space associated with the inner function $\Theta_{\ell}$ defined in \eqref{eq:def-Theta-ell}, namely
\begin{equation*}
    K^2_{\Theta_{\ell}} \defi H^2(\C_+) \ominus \Theta_{\ell} H^2(\C_+).
\end{equation*}
By \eqref{eq:Theta-ell-E}, the standard correspondence between de~Branges spaces and model spaces implies that division by $E_{\ell}$ defines an isometric isomorphism
\begin{equation}\label{eq:def-Iso-Isomet}
    \mathcal V_{\ell} : \mathcal H(E_{\ell}) \longrightarrow K^2_{\Theta_{\ell}}, \qquad \mathcal V_{\ell} F \defi \frac{F}{E_{\ell}}.
\end{equation}
In particular, since $E_{\ell}$ has no real zeros, $\V_{\ell}$ preserves real zeros. Furthermore, by the following result of Fricain and Mashreghi \cite{Fricain2008}, the derivatives of an element of $K^2_{\Theta_{\ell}}$ at a point $x_0 \in \mathbb{R}$ can be bounded pointwise in terms of the derivatives of $\Theta_{\ell}$, provided that $x_0$ is sufficiently regular for $\Theta_{\ell}$.
\begin{theoremB}[{\cite[Lemma~3.2 and Proposition~5.1]{Fricain2008}}]\label{thm:B}
Let $\Theta_{\ell}$ be the meromorphic inner function defined in \eqref{eq:def-Theta-ell}, let
$k\in \Z_+$, and let $x_0\in\R$ satisfying
\begin{equation}\label{eq:cond-derivatives-Theta}
    \sum_{n\ge1}\frac{\IM (z_n)}{|x_0-z_n|^{2k+2}}<\infty.
\end{equation}
Then there exist $\mathcal{K}^{\Theta_{\ell}}_{x_0,k} \in K^2_{\Theta_{\ell}}$, such that for every $f\in K^2_{\Theta_{\ell}}$, we have $f^{(k)}(x_0)=\langle f, \mathcal{K}^{\Theta_{\ell}}_{x_0,k}\rangle$ and 
\begin{equation*}
    |f^{(k)}(x_0)| \leq \|f\|_{K^2_{\Theta_{\ell}}} \hp{ \frac{(k!)^{2}}{2\pi} \sum_{j=0}^{k} \frac{|\Theta_{\ell}^{(j)}(x_0)|}{j!} \frac{|\Theta_{\ell}^{(2k+1-j)}(x_0)|}{(2k+1-j)!}}^{1/2}.
\end{equation*}
\end{theoremB}

\subsection[Proofs of Theorem~NUP and Corollary~Hilbert Transform (ii)]{Proofs of Theorem~\ref{thm:debranges-nup} and Corollary~\ref{cor:HT}~(ii)}\label{sec:5.4}

The main goal is to prove Theorem~\ref{thm:debranges-nup} and the strategy is to combine the two theorems from Subsection~\ref{sec:5.3} to construct a function with the prescribed zero set $\Lambda \cup M$. First, by Theorem~\hyperref[thm:A]{A}, we construct a function in a de~Branges space $\mathcal H(E)$ and the density of this zero set is related to the phase function of $E$. This function, with its prescribed zeros, is then transferred to $K^2_{\Theta_{\ell}}$ using the isomorphism that preserves zeros.

The point of working in $K^2_{\Theta_{\ell}}$ is that Theorem~\hyperref[thm:B]{B} yields pointwise derivative estimates in terms of the corresponding inner function, and that $K^2_{\Theta_{\ell}}$ is a closed subspace of $H^2(\C_+)$. This Hardy space becomes special in this context because if $f\in H^2(\C_+)\cap \Sp(\R)$, we have $\supp \widehat{f^{(k)}} \subset \R_+$ for every $k\in \Z_+$ and this implies, as above, that $(-\Delta)^{\frac{1}{2}}f(x) = \frac{1}{2\pi i} f^\prime(x)$ pointwise. In particular, constructing such a function with double zeros at $\Lambda \cup M$ and the appropriate decay yields the conclusion.

\begin{proof}[Proof of Theorem~\ref{thm:debranges-nup}]
Fix\footnote{Once $\alpha$ is fixed, we allow the relations $\asymp$ and $\lesssim$ to depend on $\alpha$, but omit this dependence to simplify the notation.} $0<\alpha<1$, and set $\beta = \alpha^{-1}$. The first step is to construct a sequence $\{z_n\} \subset \C_+$, depending on $\alpha$, such that the Blaschke product \eqref{eq:def-Theta-ell} defines, for every $\ell \in \N$, an inner function $\Theta_{\ell}$. Let $E_{\ell}$ denote the associated Hermite--Biehler function with phase $\varphi_{\alpha,\ell}$. The sequence $\{z_n\}$ will be chosen so that the phase functions satisfy the following properties:
\begin{enumerate}
\item[$(\mathsf P_1)$]\label{item:P1} $\varphi_{\alpha,\ell}'(x) \asymp \ell (1+|x|)^{\beta-1}$;
\item[$(\mathsf P_2)$]\label{item:P2} For every $k\in \N$ and $x\in \R$ \eqref{eq:cond-derivatives-Theta} holds, and $|\varphi_{\alpha,\ell}^{(k)}(x)| \lesssim_{k} \ell (1+|x|)^{k\hp{\beta-1}}$.
\end{enumerate}
We shall refer to the above conditions as property $(\mathsf{P})$. For the moment, we assume the existence of a family of phase functions $\varphi_{\alpha,\ell}$ satisfying $(\mathsf{P})$, which will be constructed in Appendix~\ref{app:A}.

Fix sequences $\Lambda,M \subset \mathbb{R}$ with $\D^+_{G_\alpha}(\Lambda)$ and $\D^+_{G_\alpha}(M)$ both finite, where $G_{\alpha}$ is defined in \eqref{eq:def-G-Phi}. Since $G_{\alpha}(\R)=\R$, we can apply the formula in \eqref{eq:D-meas-func}. In particular, consider
\begin{equation}\label{eq:densi-G-alpha-two}
    C_0=\max\hch{\D^{+}(G_{\alpha}(\Lambda)), \D^{+}(G_{\alpha}(M))}<\infty.
\end{equation}

\begin{claim}\label{claim:1}
Let $\Gamma$ denote either $\Lambda$ or $M$. If $\varphi_{\alpha, \ell}$ satisfies \hyperref[item:P1]{$(\mathsf{P}_1)$}, then there exists a constant $C>0$, depending only on $\alpha$, such that
\begin{equation*}
    \D^+_{\varphi_{\alpha, \ell}}(\Gamma)\le \frac{CC_0}{\ell}.
\end{equation*}
\end{claim}
\begin{proof}[Proof of Claim \ref{claim:1}]
Define $H_\alpha(x)=\sgn(x)((1+|x|)^\beta-1)$. Then $H_\alpha$ is increasing and for $x \in \R$
\begin{equation}\label{eq:der-H-line}
    H_\alpha'(x)=\beta(1+|x|)^{\beta-1}.
\end{equation}
We already know by \hyperref[item:P1]{$(\mathsf{P}_1)$} that $\varphi_{\alpha,\ell}(\R)=H_\alpha(\R)=\R$. Consequently, for $H_\alpha$ or $\varphi_{\alpha,\ell}$, the upper density may be written as \eqref{eq:D-meas-func}. 

Let $I=[a,b]\subset\R$ be an interval. Using \eqref{eq:der-H-line}, together with \hyperref[item:P1]{$(\mathsf{P}_1)$}, and integrating over $I$, we obtain
\begin{equation}\label{eq:H-phi-intervals}
    \ell\frac{c_1}{\beta}|H_\alpha(I)| \leq |\varphi_{\alpha, \ell}(I)|\leq \ell\frac{C_1}{\beta}|H_\alpha(I)|,
\end{equation}
for some constants $c_1,C_1>0$ independent of $\ell$. We will prove that there exist constants $C_2\ge 1$ and $C_3\ge 0$, depending only on $\alpha$, such that for every interval $I\subset\R$,
\begin{equation}\label{eq:H-G-intervals}
    |H_\alpha(I)|\leq C_2|G_\alpha(I)|+C_3, \qquad |G_\alpha(I)|\leq C_2|H_\alpha(I)|+C_3.
\end{equation}

To prove this, first consider an interval $J=[u,v]\subset [1,\infty)$.
Since $1\le \frac{1+x}{x}\le 2$ for $x\ge 1$, we have
\begin{equation*}
    2^{-|\beta-1|}x^{\beta-1}\leq (1+x)^{\beta-1}\leq 2^{|\beta-1|}x^{\beta-1}.
\end{equation*}
Multiplying by $\beta$ and integrating from $u$ to $v$, we obtain
\begin{equation*}
    2^{-|\beta-1|} |G_\alpha(J)| \leq |H_\alpha(J)|\leq 2^{|\beta-1|} |G_\alpha(J)|. 
\end{equation*}
By oddness of both $G_\alpha$ and $H_\alpha$, the same estimate holds for every interval $J\subset(-\infty,-1]$. Now let $I=[a,b]$ be arbitrary, and decompose it as
$I=I_-\cup I_0\cup I_+$, where
\begin{equation*}
    I_-=I\cap(-\infty,-1],\quad I_0=I\cap[-1,1],\quad I_+=I\cap[1,\infty).
\end{equation*}
Since $G_\alpha$ and $H_\alpha$ are increasing functions, the lengths of their images are additive over adjacent subintervals, with
\begin{equation*}
    |G_\alpha(I)|=|G_\alpha(I_-)|+|G_\alpha(I_0)|+|G_\alpha(I_+)|,\qquad |H_\alpha(I)|=|H_\alpha(I_-)|+|H_\alpha(I_0)|+|H_\alpha(I_+)|,
\end{equation*}
where empty pieces are omitted. On $I_-$ and $I_+$, we have the multiplicative comparison above. On the middle piece $I_0\subset[-1,1]$, both image lengths are uniformly bounded, because
\begin{equation*}
    |G_\alpha(I_0)|\le |G_\alpha([-1,1])|=2, \qquad |H_\alpha(I_0)|\le |H_\alpha([-1,1])|=2(2^\beta-1).  
\end{equation*}
Combining these estimates proves \eqref{eq:H-G-intervals}. 

Fix $\varepsilon>0$. Since $\D^+(G_\alpha (\Gamma))\leq C_0$ by \eqref{eq:densi-G-alpha-two}, it follows from the definition of $\limsup$ that there exists $R_\varepsilon>0$ such that every interval $L\subset\R$ with $|L|\ge R_\varepsilon$ satisfies
\begin{equation}\label{eq:ineq-dens-G-L}
    \#(G_\alpha(\Gamma)\cap L)\leq (C_0+\varepsilon)|L|.
\end{equation}
Now let $J\subset\R$ be an interval of length $|J|=r$, and define $I=\varphi_{\alpha,\ell}^{-1}(J)$. Since $\varphi_{\alpha,\ell}$ is increasing and onto $\R$, the set $I$ is an interval. Also, since $G_\alpha$ is increasing and onto $\R$, the set $K=G_\alpha(I)$ is also an interval. We estimate $|K|$ in terms of $r$ and $\ell$. From \eqref{eq:H-phi-intervals} and the second inequality in \eqref{eq:H-G-intervals}
\begin{equation}\label{eq:upp-bound-meas-K}
    |K|=|G_\alpha(I)|\le C_2|H_\alpha(I)|+C_3 \leq C_2\frac{\beta}{c_1}\frac{r}{\ell}+C_3.
\end{equation}
Similarly, using the first inequality in \eqref{eq:H-G-intervals}, we have $|H_\alpha(I)|\le C_2|K|+C_3$. Combining this with the upper bound in \eqref{eq:H-phi-intervals}, namely $r=|\varphi_{\alpha,\ell}(I)|\le \frac{C_1}{\beta}\ell |H_\alpha(I)|$, we obtain
\begin{equation}\label{eq:low-bound-meas-K}
    |K|\ge \frac{\beta}{C_2C_1}\frac{r}{\ell}-\frac{C_3}{C_2}.
\end{equation}

In particular, for each fixed $\ell$, \eqref{eq:low-bound-meas-K} shows that $|K|\to\infty$ as $r\to\infty$. Therefore, if $r$ is sufficiently large, then $|K|\ge R_\varepsilon$, and \eqref{eq:ineq-dens-G-L} applies to $K$. Next, because $G_\alpha$ is injective,
\begin{equation*}
    \#(\varphi_{\alpha,\ell}(\Gamma)\cap J) = \#(\Gamma\cap I) = \#(G_\alpha(\Gamma)\cap G_\alpha(I)) = \#(G_\alpha(\Gamma)\cap K).
\end{equation*}
Hence, for all sufficiently large $r$, $\#(\varphi_{\alpha,\ell}(\Gamma)\cap J)\le (C_0+\varepsilon)|K|$. Using \eqref{eq:upp-bound-meas-K} and $|J|=r$, we obtain
\begin{equation*}
    \frac{\#(\varphi_{\alpha,\ell}(\Gamma)\cap J)}{|J|} \leq (C_0+\varepsilon)\hp{C_2\frac{\beta}{c_1}\frac1\ell+\frac{C_3}{r}}.
\end{equation*}
Taking the supremum over all intervals $J\subset\R$ with $|J|=r$, and then taking $\limsup$, we have
\begin{equation*}
    \D_{\varphi_{\alpha,\ell}}^+(\Gamma) \le C_2\frac{\beta}{c_1}\frac{C_0+\varepsilon}{\ell}.
\end{equation*}
The conclusion is obtained in the limit as $\varepsilon \to 0$.
\end{proof}

Combining Claim~\ref{claim:1} with Remark~\ref{rem:2.7}, we obtain decompositions of $\varphi_{\alpha,\ell}(\Lambda)$ and $\varphi_{\alpha,\ell}(M)$ into finitely many pairwise disjoint uniformly discrete subsequences. Thus, there exist
\begin{equation*}
    \varphi_{\alpha,\ell}(\Lambda) = \varphi_{\alpha,\ell}(\Gamma_1) \cup \dots \cup \varphi_{\alpha,\ell}(\Gamma_N), \quad \varphi_{\alpha,\ell}(M) = \varphi_{\alpha,\ell}(\Gamma_{N+1}) \cup \dots \cup \varphi_{\alpha,\ell}(\Gamma_L),
\end{equation*}
with each union disjoint, and such that each $\varphi_{\alpha,\ell}(\Gamma_i)$ is uniformly discrete.

Let $\lambda' \in \mathbb{R} \setminus (\Lambda \cup M)$, and for each $j \in \{1,\dots,L\}$ define $\Gamma_j' = \Gamma_j \cup \{\lambda'\}$. Choose $j \in \{1,\dots,L\}$. By construction, $\Gamma_j'$ is uniformly $\varphi_{\alpha,\ell}$-discrete. Moreover, since 
\begin{equation*}
    \varphi_{\alpha,\ell}(\Gamma_j') \subset \varphi_{\alpha,\ell}(\Gamma)\cup \varphi_{\alpha,\ell}(\lambda'),
\end{equation*}
where $\Gamma$ denotes either $\Lambda$ or $M$ depending on $j$, it follows that
\begin{equation*}
    \D^+(\varphi_{\alpha, \ell}(\Gamma_j^\prime))\leq \D^+(\varphi_{\alpha, \ell}(\Gamma)).
\end{equation*}
Using Claim~\ref{claim:1}, there exists $\ell_0$ independent of $j$ such that for all $\ell \ge \ell_0$ we have $\D_{\varphi_{\alpha,\ell}}^+(\Gamma) < \pi^{-1}$. 
Consequently, for $\ell\geq \ell_0$
\begin{equation*}
    \D_{\varphi_{\alpha,\ell}}^+(\Gamma_j') < \pi^{-1}.
\end{equation*}
Furthermore, the measure $\varphi_{\alpha, \ell}^\prime(x)\dx$ is doubling. Indeed, since we already know that power weights are doubling by Example~\ref{ex:2.4}, the doubling property of $\varphi_{\alpha, \ell}^\prime(x)\dx$ follows from \hyperref[item:P1]{$(\mathsf{P}_1)$} and Remark~\ref{rem:2.5}.

We are now in a position to apply Theorem~\hyperref[thm:A]{A}, which we invoke for the sequence $\hch{a_{j,\gamma}}_{\gamma\in \Gamma_j'}$
\begin{equation*}
    a_{j,\gamma} 
    =\begin{cases}
    1, \text{ if } \gamma=\lambda^\prime, \\
    0, \text{ if } \gamma\in \Gamma_j.
    \end{cases}
\end{equation*}
In particular, there exists $g_j\in \H(E_{\ell})$ such that $g_j(\lambda^\prime)=1$ and $g_j(\gamma)=0$ for $\gamma\in\Gamma_j$. Since $E_{\ell}$ has no real zeros, by \eqref{eq:def-Iso-Isomet} we have
\begin{equation}\label{eq:def-g_j}
    f_j \defi \mathcal V_{\ell}(g_j) \in K^2_{\Theta_{\ell}}
\end{equation}
also vanishes on $\Gamma_j$. By \hyperref[item:P2]{$(\mathsf{P}_2)$}, we can apply Theorem~\hyperref[thm:B]{B} for every $k\in \Z_+$. In particular, for $f_j$ defined in \eqref{eq:def-g_j} we have
\begin{equation}\label{eq:bound-f_j}
|f_j^{(k)}(x)|^2 \lesssim_k \|f_j\|_{K^2_{\Theta_{\ell}}}^2 \sum_{r=0}^{k} \hmm{\Theta_{\ell}^{(r)}(x)} \hmm{\Theta_{\ell}^{(2k+1-r)}(x)}.
\end{equation}
Since $\Theta_{\ell}(x) = e^{2i\varphi_{\alpha,\ell}}$, for $n\in \hch{1,\dots,2k+1}$ we have $|\Theta_{\ell}^{(n)}(x)|=|P_n(\varphi_{\ell}'(x),\ldots,\varphi_{\ell}^{(n)}(x))|$ for $P_n$ a polynomial. Then there exists $c_n\in\N$ such that
\begin{equation*}
|P_n(y_1,\ldots,y_n)| \lesssim_k \hp{1+\max_{1\le j\le n}|y_j|}^{c_n}.
\end{equation*}
Using these bounds in \eqref{eq:bound-f_j}, together with \hyperref[item:P2]{$(\mathsf{P}_2)$}, we obtain for some $d_k, N_k\in\N$
\begin{equation}\label{eq:bound-derv-f_j}
    |f_j^{(k)}(x)| \lesssim_{j,k} (1+\ell)^{d_k}(1+|x|)^{N_k}.
\end{equation}
 
Let $r = \max\hch{p,q}$. Consider $0\neq \varphi$ as in Proposition~\ref{prop:phi-psi}, with $\widehat{\varphi} \in G^{1/r}(U)$, except that on the Fourier side we take $U = \hch{\xi\in \R \st 1<\xi<2}$. Define the nonzero function
\begin{equation*}
    h(x) = \varphi(x)\prod_{j=1}^L(f_j(x))^2.
\end{equation*}
Using that every $f_j$ and $\varphi$ belong to $\mathscr{H}(\R)$, the analyticity of $h$ in $\R$ follows. Since $(f_j)^2$ vanishes to order two on $\Gamma_j$, both $h$ and $h^\prime$ vanish on $M\cup \Lambda$. Moreover, because 
\begin{equation*}
   f_j\in K^2_{\Theta_{\ell}}\subset H^2(\C_+), \qquad \supp \widehat{\varphi}\subset \R_+,  
\end{equation*}
we see that $\supp \widehat{h}\subset \R_+$, with $(-\Delta)^{1/2}h(x) = \tfrac{1}{2 \pi i} h^\prime (x)$. By \eqref{eq:bound-derv-f_j} and the subexponential decay of $\varphi$ and its derivatives, we have that $h\in\mathcal{S}_{1/2}^{r,r}\subset \mathcal{S}_{1/2}^{p,q}.$
\end{proof}
Next, we prove the non-uniqueness part of Corollary~\ref{cor:HT}.
\begin{proof}[Proof of Corollary~\ref{cor:HT}~{\hyperref[item:2-HT]{$(ii)$}}]
Fix sequences $\Lambda, M \subset \mathbb{R}$ such that $\D^+_{G_\alpha}(\Lambda)$ and $\D^+_{G_\alpha}(M)$ are finite, and set $r = \max\hch{p,q}$. Let $h \in \Sp_{1/2}^{r,r}(\R)$ be the nonzero function defined in the proof of Theorem~\ref{thm:debranges-nup}. By construction $\supp \widehat{h} \subset \R_+$, which implies
\begin{equation*}
    {\bf H}(h)(x) = -i h(x).
\end{equation*}
Since $h$ is real analytic, vanishes at $\Lambda \cup M$ and $|h(x)| \lesssim e^{-c|x|^r}$ for some $0 < c < 1$, we have $h \in \Sp_{\bf H}^{r,r}(\R)\subset \Sp_{\bf H}^{p,q}(\R)$.
\end{proof}

\section{Unique continuation revisited}\label{sec:6}
We can interpret Theorems \ref{thm:UP} and \ref{thm:nup-strong} within the boundary framework introduced in Subsection~\ref{sec:2.2}. Let $f\in \Sp(\R^d)$ and $u_f : \mathbb{R}^{d+1}_+ \to \mathbb{C}$ be the smooth function that weakly vanishes as $y\to \infty$ satisfying
\begin{equation*}
    \begin{cases}
        \mathrm{div}(y^{1-2s} \nabla u_f) = 0 &\text{ in } \mathbb{R}^{d+1}_+, \\ 
        u_f(\cdot,0) = f &\text{ on } \R^d.
    \end{cases}
\end{equation*}

It follows from \eqref{eq:frac-L-limit} that the vanishing of $(-\Delta)^s f$ on a set $\Gamma \subset\R^d$ is equivalent to the vanishing of the weighted normal derivative $\lim_{y\to0^+} y^{1-2s}\partial_y u_f(\gamma,y)$ for $\gamma\in\Gamma$. Therefore, our main results about uniqueness pairs admit the following unique continuation property on the boundary for such a degenerate elliptic equation.
\begin{corollary}\label{cor:UC-extension}
Let $0<\alpha,\rho,s<1$ and $c>0$. Consider $\Phi_{\alpha,c}$ as in \eqref{eq:def-G-Phi}, and set $\beta=\frac{1}{\alpha}$.
\begin{enumerate}
    \item If the trace of $u_f$ vanishes on $\Lambda$ and $\lim_{y\to0^+} y^{1-2s}\partial_y u_f(x,y)$ vanishes on $M$, where both $M, \Lambda \subset \R^d$ are discrete, with $\D_{\Phi_{1,c}}^{-}(\Lambda)$ and $\D_{\Phi_{\beta,c}}^{-}(M)$ positive, then $u_f\equiv0$ in $\R^{d+1}_+$.
    \item Given two $\rho$-separated sets $\Lambda, M\subset \R^d$, there exists a nontrivial solution $u_f$ whose trace vanishes on $\Lambda$ and whose weighted normal derivative vanishes on $M$.
\end{enumerate}
\end{corollary}
Under suitable regularity assumptions, these results can be viewed as a discrete version of the phenomena discussed in the \hyperref[sec:1]{Introduction} for harmonic functions in $\mathbb{R}^{d+1}_+$. For dimension $d=1$, it is known, see \cite{AE1}, that uniqueness holds whenever the underlying set has positive measure, while the Bourgain--Wolff construction \cite{BourgainWolff1990} provides higher-dimensional examples of sets with positive measure where uniqueness fails. In a discrete setting, what governs this behavior is not the measure but the asymptotic growth of the set at infinity: sparse sets that grow slowly, such as $\hch{\pm c \log(n)}_{n\in \N}$, still impose $u_f \equiv 0$, while sets with faster growth at infinity, such as $\hch{\pm c(\log(n))^\beta}_{n\in \N}$ for $\beta>1$, admit nontrivial solutions.

\section{Extensions to multiplier operators}\label{sec:7}
Let $m:\mathbb{R}^d \to \mathbb{C}$ be a continuous function that is polynomially bounded. Define the operator $T_m$ on the Schwartz class as
\begin{equation}\label{eq:def-T_m}
    \widehat{T_m(f)} = m \widehat{f}.
\end{equation}
We can formulate some of the preceding uniqueness and non-uniqueness results for these more general multipliers. The corresponding statements are given below, with the details following from straightforward modifications of the previous proofs.
\begin{remark}[Theorem~\ref{thm:UP}--revisited]\label{rem:7.1}
Suppose there exists a nonempty closed set $S\subset \C^d$ such that
\begin{enumerate}
    \item $m$ is real-analytic on $\R^d\setminus S$;
    \item writing $\tau = \inf \hch{|\IM(z)|\st z\in S}$ and $\Omega_c=\hch{z\in \C^d \st |\IM(z)|<c}$, for no $c>\tau$ does there exist a meromorphic function on $\Omega_c \setminus S$ whose germ at each point $x_0 \in \R^d \setminus S$ agrees with the germ of $m$ at $x_0$.
\end{enumerate}
Let $\Lambda, M\subset \R^d$ be discrete. If there exist $0<c_1,c_2<(2\pi\tau)^{-1}$ such that $\D^-_{\Phi_{1,c_1}}(\Lambda)>0$ and $\D^-_{\Phi_{1,c_2}}(M)>0$, then $(\Lambda, M)$ is a $\mathsf{UP}_{T_m}$ for $\Sp(\R^d)$.
\end{remark}

\begin{remark}[Theorem~\ref{thm:BL-nup}--revisited]\label{rem:7.2}
Let $M\in \GL_d(\R)$, $\Lambda^*=M^{-t}\Z^d$, and $Q_M=M^{-t}[0,1)^d$. Assume that there exist distinct points $v_1,v_2,v_3\in\Lambda^*$ and a nonempty open set
$O\Subset \operatorname{int}(Q_M)$ such that $m$ is $C^\infty$ in a neighborhood of $\bigcup_{j=1}^3(\overline O+v_j)$, 
and
\begin{equation*}
    m(\xi+v_2)-m(\xi+v_1)\neq 0 \quad \text{for every } \xi\in \overline O .
\end{equation*}
Then $(M\Z^d,M\Z^d)$ is a $\mathsf{NUP}_{T_m}$ for $\Sp(\R^d)$.
\end{remark}

\subsection{Examples}
We conclude with examples of multipliers that are covered by the previous remarks.

\subsubsection{Shifted fractional Laplacian}
Let $s,\lambda> 0$ and 
\begin{equation*}
    m(\xi) = (|\xi|^2 + \lambda)^s.
\end{equation*}
The operator in \eqref{eq:def-T_m} coincides with the fractional power of the shifted Laplacian $T_m = (-\Delta + \lambda)^s$. In this setting, Remark~\ref{rem:7.2} applies. In addition, if $s \in \R_{>0} \setminus \mathbb{N}$, then Remark~\ref{rem:7.1} also holds.

\subsubsection{Unimodular multipliers}
Let $\lambda\geq 0$, $s>0$, and consider the Fourier multiplier
\begin{equation*}
    m(\xi) = e^{i(|\xi|^2 + \lambda)^{s}}.
\end{equation*}
The case $s=\frac{1}{2}$, $\lambda=0$ corresponds to the time evolution of the wave equation, while $s=1$, $\lambda=0$ yields the free Schrödinger propagator, both at time $t=1$. The choice $s=\frac{1}{2}$ and $\lambda=1$ is associated with the solution of the Klein--Gordon equation, also at time $t=1$. Operators of this type have been studied with respect to their boundedness properties, see for instance \cite{Benyi2007}. In our context, if $s \in \R_{>0} \setminus \mathbb{N}$, then Remark~\ref{rem:7.1} applies.

\subsubsection{Mixed multipliers}
Let $s,\gamma > 0$, and
\begin{equation*}
    m(\xi) = \frac{e^{i|\xi|^{2s}}}{(1 + |\xi|^2)^{\gamma/2}}.
\end{equation*}
This class combines oscillatory behavior with polynomial decay and has appeared in different contexts, see \cite{Hirschman59}. In the present framework, Remark~\ref{rem:7.2} is applicable. Furthermore, when $s \in \R_{>0} \setminus \mathbb{N}$, Remark~\ref{rem:7.1} is also valid.

\appendix

\section{Proof of property \texorpdfstring{$(\mathsf{P})$}{(P)}}\label{app:A}
The following proposition is based on a construction due to Baranov, Borichev, and Havin \cite[Lemma~4.2]{BaranovHavin}. In that work, the case $\frac{1}{2} < \alpha < 1$ is treated in connection with a different problem. For completeness, we provide a proof for $0 < \alpha < 1$. Moreover, throughout this appendix, the relations $\asymp$ and $\lesssim$ are allowed to depend on $\alpha$, and this dependence is omitted from the notation.
\begin{proposition}
Let $0<\alpha <1$, $\ell\in \N$, and for $n\in \N$ define
\begin{equation}\label{eq:A.1}
x_n=n^\alpha, 
\qquad
y_n= \begin{cases}
        n^{\alpha-1}, & 0<\alpha\leq \frac12,\\
        1, & \frac12<\alpha<1,
    \end{cases}
\end{equation}

\begin{equation*}
    z_{2n-1}\defi x_n+iy_n,\qquad z_{2n}\defi-x_n+iy_n.
\end{equation*}
Then the product in \eqref{eq:def-Theta-ell} defines an inner function. Let $E_{\ell}$ denote the associated Hermite--Biehler function with phase $\varphi_{\alpha,\ell}$. Set $\beta = \alpha^{-1}$. For every $k\in \N$ and $x\in \R$, we have 
\begin{equation}\label{eq:A.2}
    |\varphi_{\alpha,\ell}^{(k)}(x)|\lesssim_{k} \ell (1+|x|)^{k\hp{\beta-1}}.
\end{equation}
Furthermore, \eqref{eq:cond-derivatives-Theta} holds and
\begin{equation}\label{eq:A.3}
    \varphi_{\alpha, \ell}'(x)\asymp \ell (1+|x|)^{\beta-1}.
\end{equation}
\end{proposition}
\begin{proof}
First of all, it is clear from \eqref{eq:A.1} that \eqref{eq:def-Theta-ell} is an inner function, given that the Blaschke condition \eqref{eq:def-Blasc-cond} is verified for the zero set.

We begin by proving \eqref{eq:A.2} and \eqref{eq:cond-derivatives-Theta}. By the mean value theorem, for each $n\geq 1$ there exists $\xi_n\in(n,n+1)$ such that
\begin{equation*}
    \Delta_n\defi x_{n+1}-x_n=(n+1)^\alpha-n^\alpha=\alpha \xi_n^{\alpha-1}.
\end{equation*}
Since $\xi_n\asymp n$, it follows that
\begin{equation}\label{eq:A.4}
    \Delta_n\asymp n^{\alpha-1}.
\end{equation}
Consequently, for $n\asymp N$
\begin{equation}\label{eq:A.5}
    \Delta_n\asymp \Delta_N, \qquad y_n\asymp y_N.
\end{equation}
In particular,
\begin{equation}\label{eq:A.6}
    \Delta_N^{-1}\asymp N^{1-\alpha}.
\end{equation}
We next recall the representation of the derivatives of $\varphi_{\alpha, \ell}$. Let $P_v(u)$ denote the (non-normalized) Poisson kernel for $\mathbb{R}^2_+$, defined by
\begin{equation*}
    P_v(u)=\frac{v}{u^2+v^2} \quad \text{for } u\in\mathbb{R},\, v>0.
\end{equation*}
Using the formula \eqref{eq:Der-Theta-ell} for $\varphi_{\alpha, \ell}'$, we have
\begin{equation}\label{eq:A.7}
    \varphi_{\alpha, \ell}'(x)= \ell \sum_{n=1}^\infty \hp{P_{y_n}(x-x_n)+P_{y_n}(x+x_n)}.
\end{equation}
We now justify differentiation to arbitrary order. For each integer $m\geq 0$,
\begin{equation}\label{eq:A.8}
    |P_v^{(m)}(u)| \lesssim_m\frac{v}{(u^2+v^2)^{1+\frac{m}{2}}}.
\end{equation}

Fix now $k\geq 1$ and let $K\subset\R$ be compact. For all $n$ sufficiently large, the estimate \eqref{eq:A.8} with $m=k-1$ gives
\begin{equation*}
    |P_{y_n}^{(k-1)}(x\pm x_n)| \lesssim_k \frac{y_n}{x_n^{k+1}}
\end{equation*}
uniformly in $x\in K$. Now
\begin{equation*}
    \frac{y_n}{x_n^{k+1}} = 
    \begin{cases} 
    n^{\alpha-1-\alpha(k+1)}=n^{-1-\alpha k}, & 0<\alpha\leq \frac12,\\ 
    n^{-\alpha(k+1)}, & \frac12<\alpha<1, 
    \end{cases}
\end{equation*}
and both series are summable. Hence, using the Weierstrass $M$-test, we conclude that
\begin{equation*}
    \sum_{n=1}^\infty \hp{P_{y_n}^{(k-1)}(x-x_n)+P_{y_n}^{(k-1)}(x+x_n)}
\end{equation*}
converges uniformly on compact subsets of $\R$. Therefore, by the standard theorem on termwise differentiation of uniformly convergent series of smooth functions
\begin{equation}\label{eq:A.9}
    \varphi_{\alpha, \ell}^{(k)}(x) = \ell \sum_{n=1}^\infty \hp{P_{y_n}^{(k-1)}(x-x_n)+P_{y_n}^{(k-1)}(x+x_n)}.
\end{equation}

Applying \eqref{eq:A.8} with $m=k-1$, we have from \eqref{eq:A.9} that $|\varphi_{\alpha, \ell}^{(k)}(x)| \lesssim_k \ell S_k(x)$, where
\begin{equation*}
    S_k(x)\defi\sum_{n=1}^\infty\hp{\frac{y_n}{\hp{(x-x_n)^2+y_n^2}^{(k+1)/2}} + \frac{y_n}{\hp{(x+x_n)^2+y_n^2}^{(k+1)/2}}}.
\end{equation*}
Set $s=\frac{k+1}{2}$. Then $s>\frac12$, and we may rewrite
\begin{equation}\label{eq:A.10}
    S_k(x)=S_k^-(x)+S_k^+(x),
\end{equation}
where
\begin{equation*}
    S_k^-(x)=\sum_{n=1}^\infty \frac{y_n}{\hp{(x-x_n)^2+y_n^2}^s}, \qquad S_k^+(x)=\sum_{n=1}^\infty \frac{y_n}{\hp{(x+x_n)^2+y_n^2}^s}.
\end{equation*}
Since the left-hand side of \eqref{eq:A.10} defines an even function, it suffices to consider $x\geq 0$. In the compact interval $[0,2]$, the function $\varphi_{\alpha, \ell}^{(k)}$ is continuous, therefore bounded, and consequently \eqref{eq:A.2} is immediately obtained by \eqref{eq:A.9} after enlarging the constant if necessary.

We may assume from now on that $x\geq 2$, and we choose the unique integer $N$ such that
\begin{equation*}
    x_N\leq x<x_{N+1}.
\end{equation*}
Since $x_N=N^\alpha$, we have
\begin{equation}\label{eq:A.11}
    N\asymp x^{\beta}, \qquad \Delta_N^{-1}\asymp N^{1-\alpha}\asymp x^{\beta-1}.
\end{equation}

\begin{claim*}
If $s>\frac12$, then
\begin{equation}\label{eq:A.12}
    \sum_{N/2<n<4N}\frac{y_n}{\hp{(x-x_n)^2+y_n^2}^s} \lesssim y_N^{1-2s}\hp{1+\frac{y_N}{\Delta_N}}.
\end{equation}
\end{claim*}
\begin{proof}
By \eqref{eq:A.5}, one has $y_n\asymp y_N$ throughout the range $N/2<n<4N$. Next, if $n\leq N-1$, then
\begin{equation*}
    x-x_n\geq x_N-x_n=\sum_{j=n}^{N-1}\Delta_j\gtrsim (N-n)\Delta_N,
\end{equation*}
where we used \eqref{eq:A.5} to compare $\Delta_j$ with $\Delta_N$ for $j\asymp N$. Likewise, if $n\geq N+2$, then
\begin{equation*}
    x_n-x\geq x_n-x_{N+1}=\sum_{j=N+1}^{n-1}\Delta_j\gtrsim (n-N-1)\Delta_N.
\end{equation*}
Hence
\begin{equation*}
    \sum_{N/2<n<4N}\frac{y_n}{\hp{(x-x_n)^2+y_n^2}^s} \lesssim y_N \sum_{m\in\Z}\frac{1}{(m^2\Delta_N^2+y_N^2)^s}.
\end{equation*}
Factoring out $y_N^{2s}$ gives
\begin{equation*}
    y_N^{1-2s} \sum_{m\in\Z} \frac{1}{\hp{1+m^2(\Delta_N/y_N)^2}^s}.
\end{equation*}
Now for every $a>0$ and every $s>\frac12$, $\sum_{m\in\Z}(1+a^2m^2)^{-s}\lesssim 1+a^{-1}$, by comparison with the convergent integral $\int_0^\infty (1+t^2)^{-s}\dt$. Applying this with $a=\Delta_N/y_N$ proves \eqref{eq:A.12} and the claim follows. 
\end{proof}

We now estimate $S_k^-(x)$. We split the sum into three regions.
First, if $n\leq N/2$, then
\begin{equation*}
    x-x_n\geq x_N-x_{\lceil N/2 \rceil}\gtrsim N^\alpha.
\end{equation*}
Therefore
\begin{equation*}
    \sum_{n\leq N/2}\frac{y_n}{\hp{(x-x_n)^2+y_n^2}^s} \lesssim N^{-2\alpha s}\sum_{n\leq N/2} y_n.
\end{equation*}
We estimate the partial sums of $y_n$ separately in the two regimes. For $0<\alpha\leq \frac12$, we have that $\sum_{n\leq N} y_n = \sum_{n\leq N} n^{\alpha-1} \lesssim N^\alpha$ by comparison with $\int_1^N t^{\alpha-1}\dt$. If $\frac12<\alpha<1$, then $y_n=1$, and therefore $\sum_{n\leq N} y_n=N$. Thus
\begin{equation*}
\sum_{n\leq N/2}\frac{y_n}{\hp{(x-x_n)^2+y_n^2}^s} 
\lesssim 
\begin{cases} 
    N^{-2\alpha s}N^\alpha=N^{-\alpha k}, & 0<\alpha\leq \frac12,\\ 
    N^{-2\alpha s}N=N^{1-\alpha(k+1)}, & \frac12<\alpha<1. 
\end{cases}
\end{equation*}
Since $\Delta_N^{-k}\asymp N^{k(1-\alpha)}$ by \eqref{eq:A.6}, and the preceding powers of $N$ are no larger than a fixed constant times $N^{k(1-\alpha)}$ for $N\geq 1$, we obtain
\begin{equation}\label{eq:A.13}
    \sum_{n\leq N/2}\frac{y_n}{\hp{(x-x_n)^2+y_n^2}^s} \lesssim \Delta_N^{-k}.
\end{equation}
Second, if $n\geq 4N$, then $x_n-x\gtrsim x_n$, because $x<x_{N+1}= (N+1)^\alpha$ and $x_n=n^\alpha$ with $n\geq 4N$. Hence
\begin{equation*}
    \sum_{n\geq 4N}\frac{y_n}{\hp{(x-x_n)^2+y_n^2}^s} \lesssim \sum_{n\geq 4N}\frac{y_n}{x_n^{2s}} = \sum_{n\geq 4N}\frac{y_n}{x_n^{k+1}}.
\end{equation*}
Thus 
\begin{equation*}
\sum_{n\geq 4N}\frac{y_n}{\hp{(x-x_n)^2+y_n^2}^s} \lesssim 
\begin{cases} 
    \sum_{n\geq 4N} n^{-1-\alpha k}\lesssim N^{-\alpha k}, & 0<\alpha\leq \frac12,\\ 
    \sum_{n\geq 4N} n^{-\alpha(k+1)}\lesssim N^{1-\alpha(k+1)}, & \frac12<\alpha<1. 
\end{cases}
\end{equation*}
Again these quantities are bounded by a constant multiple of $\Delta_N^{-k}$, and therefore
\begin{equation}\label{eq:A.14}
    \sum_{n\geq 4N}\frac{y_n}{\hp{(x-x_n)^2+y_n^2}^s} \lesssim \Delta_N^{-k}.
\end{equation}
Third, for the middle block $N/2<n<4N$, the claim \eqref{eq:A.12} yields
\begin{equation*}
    \sum_{N/2<n<4N}\frac{y_n}{\hp{(x-x_n)^2+y_n^2}^s}\lesssim y_N^{1-2s}\hp{1+\frac{y_N}{\Delta_N}}\lesssim y_N^{-k}\hp{1+\frac{y_N}{\Delta_N}},
\end{equation*}
because $2s-1=k$. We now compare this with $\Delta_N^{-k}$. If $0<\alpha\leq \frac12$, then $y_N\asymp \Delta_N$ by \eqref{eq:A.4}, and so
\begin{equation*}
    y_N^{-k}\hp{1+\frac{y_N}{\Delta_N}}\lesssim \Delta_N^{-k}.
\end{equation*}
If $\frac12<\alpha<1$, then $y_N=1$, and therefore
\begin{equation*}
    y_N^{-k}\hp{1+\frac{y_N}{\Delta_N}}=1+\Delta_N^{-1}\lesssim \Delta_N^{-k},
\end{equation*}
because $k\geq 1$ and $\Delta_N\leq 1$ for all sufficiently large $N$ (finitely many remaining values of $N$ are absorbed into the implicit constant). Hence
\begin{equation}\label{eq:A.15}
    \sum_{N/2<n<4N}\frac{y_n}{\hp{(x-x_n)^2+y_n^2}^s} \lesssim \Delta_N^{-k}.
\end{equation}
Combining \eqref{eq:A.13}, \eqref{eq:A.14}, and \eqref{eq:A.15}, we obtain
\begin{equation}\label{eq:A.16}
    S_k^-(x)\lesssim \Delta_N^{-k}.
\end{equation}

We estimate $S_k^+(x)$ in the same spirit. In contrast with the analysis of $S_k^-$, here it is enough to split the sum into the two ranges $n\leq N$ and $n>N$. Indeed, since $x,x_n>0$, we always have
\begin{equation*}
    x+x_n\ge x, \qquad x+x_n\ge x_n.
\end{equation*}
We use these bounds separately according to the range of $n$. If $n\leq N$, then $x+x_n\geq x\asymp N^\alpha$, so
\begin{equation*}
    \sum_{n\leq N}\frac{y_n}{\hp{(x+x_n)^2+y_n^2}^s} \lesssim N^{-2\alpha s}\sum_{n\leq N} y_n \lesssim \Delta_N^{-k},
\end{equation*}
by the same argument used in the estimate of the left tail. If $n>N$, then $x+x_n\geq x_n$, and therefore
\begin{equation*}
    \sum_{n>N}\frac{y_n}    {\hp{(x+x_n)^2+y_n^2}^s}\leq\sum_{n>N}\frac{y_n}{x_n^{2s}}=\sum_{n>N}\frac{y_n}{x_n^{k+1}}\lesssim\Delta_N^{-k},
\end{equation*}
exactly as in the estimate of the right tail. Hence
\begin{equation}\label{eq:A.17}
    S_k^+(x)\lesssim \Delta_N^{-k}.
\end{equation}
From \eqref{eq:A.16} and \eqref{eq:A.17} we conclude that $S_k(x)\lesssim \Delta_N^{-k}$. In particular, \eqref{eq:cond-derivatives-Theta} follows. Moreover, we have $|\varphi_{\alpha, \ell}^{(k)}(x)| \lesssim_k \ell \Delta_N^{-k}$. Using \eqref{eq:A.11}, we finally obtain $|\varphi_{\alpha, \ell}^{(k)}(x)| \lesssim_k \ell x^{k(\beta-1)}$ for all $x\geq 2$. As noted earlier, the same upper bound is trivial on $[-2,2]$ after adjusting the constant. This finishes the proof of \eqref{eq:A.2}.

For the first derivative lower bound, choose $M=\hfl{ c\frac{y_N}{\Delta_N}}+1$, where $c>0$ is small enough that $1\leq M\le N/4$ for all large $N$ (this is indeed possible because $y_N/\Delta_N\asymp 1$ if $0<\alpha\le \frac12$, while $y_N/\Delta_N\asymp N^{1-\alpha}=o(N)$ if $\frac12<\alpha<1$). For $0\le m\le M$,
\begin{equation*}
    0\le x-x_{N-m}\le x_{N+1}-x_{N-m} =\sum_{j=N-m}^{N}\Delta_j \lesssim (m+1)\Delta_N \lesssim y_N.
\end{equation*}
Moreover, since $0\leq m\leq M\leq N/4$, we have
\begin{equation*}
    \frac{3N}{4}\leq N-m\leq N.
\end{equation*}
Hence, by \eqref{eq:A.5}, for every $0\leq m\leq M$ we have $y_{N-m}\asymp y_N$, and
\begin{equation*}
    \frac{y_{N-m}}{(x-x_{N-m})^2+y_{N-m}^2}\gtrsim \frac{1}{y_N}.
\end{equation*}
Summing over $m$ throughout this range and using that, in \eqref{eq:A.7}, the summands are nonnegative, we obtain
\begin{equation*}
    \varphi_{\alpha,\ell}'(x)\ge \ell \sum_{m=0}^{M} \frac{y_{N-m}}{(x-x_{N-m})^2+y_{N-m}^2} \gtrsim\frac{M+1}{y_N} \ell \gtrsim \hp{\frac{1}{y_N} + \frac{1}{\Delta_N}} \ell \gtrsim \Delta_N^{-1}\ell. 
\end{equation*}
Together with \eqref{eq:A.2} and \eqref{eq:A.11}, this proves $\varphi_{\alpha,\ell}'(x)\asymp \ell x^{\beta-1}$ for $x \geq 2$. Since the right-hand side of \eqref{eq:A.7} is continuous and strictly positive on $[-2,2]$, we conclude \eqref{eq:A.3} and the result follows.
\end{proof}

\section*{Acknowledgments}
I would like to thank my advisors, Luz Roncal and Mateus Sousa, for many valuable discussions and for their important suggestions that significantly improved this work.

\end{document}